\documentclass[10pt]{amsart}

\usepackage{amsmath,amssymb}

\newcommand{\be}{\begin{equation}}
\newcommand{\ee}{\end{equation}}
\newcommand{\ba}{\begin{array}}
\newcommand{\ea}{\end{array}}
\newcommand{\bea}{\begin{eqnarray}}
\newcommand{\eea}{\end{eqnarray}}
\newcommand{\bee}{\begin{eqnarray*}}
\newcommand{\eee}{\end{eqnarray*}}

\newtheorem{Prop}{Proposition}

\newtheorem{Rk}{Remark}

\catcode`@=11
\renewcommand\appendix{\bigskip {\noindent\Large \bf Appendix}\par
  \setcounter{section}{0}%
  \setcounter{subsection}{0}%
  \renewcommand\thesection{\@Alph\c@section}}
\catcode`@=12

\newcommand{\beeq}{\begin{equation}}
\newcommand{\eneq}{\end{equation}}

\newcommand{\bear}{\begin{eqnarray}}
\newcommand{\eear}{\end{eqnarray}}
\newcommand{\beq}{\begin{equation}}
\newcommand{\eeq}{\end{equation}}

\newcommand{\R}{{\mathbb R}}
\newcommand{\C}{{\mathbb C}}

\newcommand{\calg}{\,{\mathfrak g}}

\newcommand{\calH}{{\mathcal H}}

\newcommand{\calN}{{\mathcal N}}

\newcommand{\la}{\langle}
\newcommand{\ra}{\rangle}

\def\pr{\partial}

\def\calge1{\calg_{\vec{e_1}}}

\def\bm{\left( \begin{array}{cc}}
\def\endm{\end{array}\right)}

\def\e{\epsilon}

\def\norm[#1][#2]{\Vert #1 \Vert_{#2}}

\def\Util3{\tilde{U}^{(3)}}

\numberwithin{equation}{section}

\newcommand{\DD}{\Delta}
\renewcommand{\to}{\rightarrow}

\newcommand{\RR}{\mathbb{R}}

\newcommand{\Honer}{H^1_{\rm r}(\RR^4)}

\newcommand{\Qt}{Q^{(t)}}
\newcommand{\Qinf}{Q^{(\infty)}}
\newcommand{\Vt}{V^{(t)}}
\newcommand{\Vinf}{V^{(\infty)}}
\newcommand{\Wt}{W^{(t)}}

\newcommand{\Ht}{\calH^{(t)}}
\newcommand{\Hinf}{\calH^{(\infty)}}
\newcommand{\Pt}{P^{(t)}}
\newcommand{\Pinf}{P^{(\infty)}}
\newcommand{\Prt}{P^{(t)}_{\rm r}}
\newcommand{\Prinf}{P^{(\infty)}_{\rm r}}

\newtheorem{theorem}{Theorem}
\newtheorem{lemma}{Lemma}

\newtheorem*{remark*}{Remark}
\newtheorem*{remarks*}{Remarks}

\newtheorem{prop}{Proposition}

\numberwithin{lemma}{section}
\numberwithin{prop}{section}

\begin{document}

\title[On Stability of Pseudo-Conformal Blowup for Hartree NLS]{On Stability of Pseudo-Conformal Blowup for $L^2$-critical Hartree NLS}

\date{July 8, 2008}

\author{Joachim Krieger} \address{University of Pennsylvania, Department of Mathematics, 4N67 Rittenhouse Lab, 209 South 33rd Street, Philadelphia, PA 19104, USA.} 
\email{krigerj@math.upenn.edu}

\author{Enno Lenzmann} 
\address{Massachusetts Institute of Technology, Department of Mathematics, Building 2, Office 230, 77 Massachusetts Avenue, Cambridge, MA 02139, USA.}
\email{lenzmann@math.mit.edu}

\author{Pierre Rapha\"el}
\address{Universit\'e Paul Sabatier Toulouse, Institut de Mathematiques, 31062 Toulouse Cedex 9, France.}
\email{pierre.raphael@math.univ-tlse.fr}

\maketitle

\begin{abstract}
We consider $L^2$-critical focusing nonlinear Schr\"odinger equations with Hartree type nonlinearity  
$$i \pr_t u = -\DD u - \big ( \Phi \ast |u|^2 \big ) u \quad \mbox{in $\RR^4$},$$
where $\Phi(x)$ is a perturbation of the convolution kernel $|x|^{-2}$. Despite the lack of pseudo-conformal invariance for this equation, we prove the existence of critical mass finite-time blowup solutions $u(t,x)$ that exhibit the pseudo-conformal blowup rate
\[
\| \nabla u(t) \|_{L^2_x} \sim \frac{1}{|t|} \quad \mbox{as} \quad t \nearrow 0 .
\] 
Furthermore, we prove the finite-codimensional stability of this conformal blow up, by extending the nonlinear wave operator construction by Bourgain and Wang (see \cite{Bourgain+Wang1997}) to $L^2$-critical Hartree NLS.
\end{abstract}


\section{Introduction}



\subsection{Setting of the problem}


Nonlinear Schr\"odinger equations (NLS) with Hartree type nonlinearity
$$\ \  \left   \{ \begin{array}{lll}
         i \pr_t u + \Delta u+(\Phi\ast |u|^2)u=0,\\[1ex]
         (t,x)\in [0,T)\times \R^d, \  \ u(0,x)=u_0(x), \ \ u_0:\R^d\to \C,
         \end{array}
\right .
$$
arise naturally as effective evolution equations in the mean-field limit of many-body quantum systems; see, e.\,g., \cite{Froehlich+Lenzmann2005} for a general overview. An essential feature of Hartree NLS is that the convolution kernel $\Phi(x)$ still retains the fine structure of microscopic two-body interactions of the quantum system. By contrast, NLS with local nonlinearities (e.\,g.~the Gross-Pitaevski equation) arise in further limiting regimes where two-body interactions are (more coarsely) modeled by a single real parameter in terms of the scattering length. In particular, NLS with local nonlinearities cannot provide effective models for quantum systems with {\em long-range} interactions such as the physically important case of the Coulomb potential $\Phi(x) \sim |x|^{-1}$ in $d=3$, whose scattering length is infinite. Moreover, such slowly decaying convolution kernels lead to long-range effects in blowup and scattering problems for Hartree NLS, which cannot be addressed by merely adapting techniques developed for local NLS. The present paper is intended to serve as a starting point for the blowup analysis of Hartree NLS.

As mentioned above, the convolution kernel $\Phi(x)=\frac{1}{|x|}$ in dimension $d=3$ represents Coulomb interactions and it is therefore of considerable physical relevance. Recently, the so-called {\em pseudo-relativistic Hartree equation}
\be
\label{relhartree}
 i \pr_t u - \sqrt{-\Delta +m^2} \, u+(\frac{1}{|x|}\ast |u|^2)u=0
 \ee
has been introduced as a relativistic correction to the classical model for the evolution of boson stars; see \cite{Elgart+Schlein2007}. Moreover, Fr\"ohlich and Lenzmann \cite{Froehlich+Lenzmann2007} have proven the existence of finite time blow up solutions for this problem in connection with the Chandrasekhar theory of gravitational collapse. Their proof, however, is based on a viriel type argument and provides no insight into the description of the singularity formation. It would be of considerable interest to extend the analysis of singularity formation for solutions to (\ref{relhartree}). However, from the mathematical point of view, this evolution equation is an $L^2$-critical blowup problem with both nonlocal dispersion and nonlocal nonlinearity, which makes its rigorous study a delicate problem.\\ 

In this paper, we propose a preliminary investigation of the singularity formation for a problem of a similar $L^2$-critical type but with more symmetries: The four dimensional $L^2$-critical Hartree NLS
\be
\label{eq:hartree} \tag{NLS$_\Phi$}
 \left   \{ \begin{array}{lll}
        i \pr_t u + \Delta u+(\Phi\ast |u|^2)u=0, \ \ \Phi(x)\sim\frac{1}{|x|^2},\\[1ex]
         (t,x)\in [0,T)\times \R^4, \  \ u(0,x)=u_0(x), \ \ u_0:\R^4\to \C.
         \end{array}
\right .
\ee
Our aim is to derive some qualitative information on possible blowup regimes. Note that the formal proximity between the relativistic three dimensional problem and its classical four dimensional version was already central in the analysis of a related problem for the gravitational Vlasov equation in astrophysics; see \cite{LMR2007}.\\ 

Let us recall the main know facts about (\ref{eq:hartree}) for the Newtonian potential $\Phi(x)=\frac{1}{|x|^2}$ in $d=4$ dimensions. The Cauchy problem is well-posed and subcritical in $H^1(\R^4)$; see, e.\,g., \cite{Ginibre+Velo1980,Cazenave2003}. Thus, for any initial datum $u_0\in H^1(\R^4)$, there exists $0<T\leq +\infty$ such that $u(t)\in C_t^0 H^1_x([0,T) \times \R^4)$, and we have either $T=+\infty$ and the solution is global, or $T<+\infty$ and then $\lim_{t\uparrow T}|\nabla u(t)|_{L^2}=+\infty$, i.\,e., the solution blows up in finite time. Furthermore, the following quantities are conserved by the $H^1$-flow:
$$         \mbox{$L^2$-norm}: \ \ \int|u(t,x)|^2=\int|u_0(x)|^2,$$
        $$ \mbox{Energy}:\ \ E(u(t,x))=\frac{1}{2}\int|\nabla u(t,x)|^2-\frac{1}{4}\int |u(t,x)|^2(\frac{1}{|x|^2}\ast |u(t,x)|^2) =E(u_0).$$
  The existence of finite time blowup solutions follows from the classical virial identity 
  $$\frac{d^2}{dt^2}\int |x|^2| |u(t,x)|^2 =16E(u_0), $$ 
  which implies finite time blow up for initial data $u_0 \in \Sigma = H^1(\RR^4) \cap L^2(\RR^4,|x|^2 \,dx)$ with negative energy $E(u_0)<0$.\\
  On the other hand, by following Weinstein \cite{Weinstein1983}, we can derive a sharp global well-posedness criterion: For any $u_0\in H^1(\R^4)$ with $\|u_0\|_{L^2_x}<\|Q\|_{L^2_x}$, the solution is global and bounded in $H^1$. Here $Q$ is the unique radially symmetric positive solution to 
\be
\label{eqsoliton}
\Delta Q +(\frac{1}{|x|^2}\ast|Q|^2)Q= Q, \ \ Q(r)>0, \ \ Q\in H^1(\R^4).
\ee
Note that the existence and uniqueness of the ground state has been proved by Lieb \cite{Lieb1977} in dimension $d=3$, but the proof can be adapted to dimension $d=4$; see Section \ref{sec:spectral} for more details. The variational characterization of $Q$ then implies the sharp interpolation estimate: 
$$\forall u\in H^1(\RR^4), \ \ E(u)\geq \frac{1}{2} \| \nabla u \|_{L^2_x}^2\left(1-\frac{\|u\|_{L^2_x}^2}{\|Q\|^2_{L^2_x}}\right),$$ 
whence the global well-posedness of $H^1$ data with $\|u_0\|_{L^2_x}<\|Q\|_{L^2_x}.$ Moreover, the sharpness of this criterion follows from the existence of the pseudo-conformal symmetry: If $u(t,x)$ solves (\ref{eq:hartree}) with $\Phi(x)=\frac{1}{|x|^2}$, then so does: $$u(t,x)=\frac{1}{t^2}\overline{u}(\frac{1}{t}, \frac{x}{t})e^{i\frac{|x|^2}{4t}}.$$ By applying this transformation to the solitary wave $u(t,x)=Q(x)e^{it}$ and using the time reflection symmetry, we obtain the critical mass blow up solution
\be
\label{st}
S(t,x)=\frac{1}{t^2}Q(\frac{x}{t})e^{\frac{i}{t}-i\frac{|x|^2}{4t}},
\ee 
which blows up according to $\| \nabla S(t) \|_{L^2_x} \sim |t|^{-1}$ as $t \nearrow 0$.
This structure is, of course, reminiscent to the one for NLS with $L^2$-critical local nonlinearity
\be
\label{NLS}
i\partial_tu+\Delta u+u|u|^{\frac{4}{d}}=0, \ \ \ (t,x)\in \R\times\R^d,
\ee 
which possess a conformal invariance and an explicit critical mass blow up solution. 


\subsection{Statement of Main Results}


The first question we ask is the persistence of the critical mass blowup solution under a deformation of the convolution kernel, which destroys the conformal invariance. Note that the question of the existence of a critical blowup element is not well understood even for local nonlinearities. In fact, it can be proven that such elements do not exist in some situations; see Martel, Merle \cite{Martel+Merle2002} for the critical KdV problem, and Merle \cite{Merle1996} for non-existence results for anisotropic nonlinearities. On the other hand, Burq-Gerard-Tzvetkov \cite{BGT2003} have shown the persistence of the critical mass blowup solution for the local (NLS) 
on a domain with Dirichlet boundary condition. Here the pseudo-conformal transformation is destroyed, but only up to an exponentially small in time term. Our first claim is that critical blowup elements persist under a small enough {\it polynomial} deformation of the pseudo-conformal symmetry. The precise statement reads as follows.

\begin{theorem}{\bf (Existence of critical mass blow up solutions).}
 \label{th:main}
Consider \eqref{eq:hartree} with $\Phi$ of the form
\begin{equation*}
\Phi(|x|) = \frac{\phi(|x|^k)}{|x|^2}, 
\end{equation*}
for some $k >0$. Here we assume that $\phi : [0,\infty ) \to \RR$ is a differentiable function such that $\phi(0) = 1$ and  $|\phi(r)| + \langle r \rangle |\phi'(r)| \leq C$ for some constant $C > 0$.

Then, for $k > 0$ sufficiently large, there exists a solution $u \in C^0_t H^1_x((-T,0) \times \RR^4)$ of \eqref{eq:hartree} with some $T > 0$ such that
\[
\|u(t)\|_{L^2_x}=\|Q\|_{L^2_x}, \ \ \| \nabla u(t) \|_{L^2_x} \sim \frac{1}{|t|} \quad \mbox{as} \quad t \nearrow 0,
\]
where $Q \in H^1(\R^4)$ is the ground state solution to (\ref{eqsoliton}).
\end{theorem}

\noindent
{\bf Comments on Theorem \ref{th:main}}\\

1. {\it Structure of the solution}: From the proof, the structure of the critical mass blowup solution is explicit and is seen to converge in some suitable sense to $S(t)$ given by (\ref{st}). Moreover, our proof is very robust and we expect it to carry over to a large class of problems, provided a certain spectral assumption can be verified.\\

2. {\it Uniqueness}: Merle proved the uniqueness of the critical mass blow up solution for the local NLS (\ref{NLS}); see \cite{Merle1993}. The proof, however, is very much based again on the existence of the pseudo-conformal symmetry. The same proof would yield uniqueness of the critical mass blow up solution for $\Phi(x)=\frac{1}{|x|^2}$, see \cite{LMR2007} for a similar result. In the more general setting of Theorem \ref{th:main}, a weak uniqueness statement could be derived simply from the fact that the solution is build by Picard iteration, but a strong general $H^1$ uniqueness statement following \cite{Merle1993} is open. This question is connected to the uniqueness of nondispersive objects, see \cite{Martel2005} for a related problem.\\

The second question we ask is the persistence of the critical type blowup regime. Here we work for the sake of simplicity directly with $\Phi(x)=\frac{1}{|x|^2}$. We adapt the analysis of Bourgain and Wang \cite{Bourgain+Wang1997}  who proved some finite codimensional stability of the $S(t)$ dynamics for the local (NLS) in space dimension $d=1,2$. 

\begin{theorem}{\bf (Finite codimensional stability of the $S(t)$ dynamics).}
 \label{th:BW}
Consider $\eqref{eq:hartree}$ with $\Phi(x)=\frac{1}{|x|^2}.$ Let $\psi_0 \in C^\infty_0(\RR^4)$ be radial, suppose $|\psi_0(x)| \lesssim |x|^N$ for $N$ sufficiently large, and define $\psi(x) = \alpha \psi_0(x)$. Then for $|\alpha| >0 $ and $\delta > 0$ sufficiently small, there exists a blowup solution
\[
u = S(t) + z_\psi+\e
\]
solving \eqref{eq:hartree} on the time interval $[-\delta, 0)$ such that $$\lim_{t \nearrow 0} \| \epsilon(t) \|_{H^1_x} = 0$$ and where $z_\psi \in C^0_t H^1_x([-\delta, +\delta] \times \RR^4)$ solves the initial-value problem
\[
\left \{ \begin{array}{l} i \partial_t z_\psi + \Delta z + \big ( |x|^{-2} \ast |z_\psi|^2 \big ) z_\psi = 0, \\[1ex]
z_\psi(0,x) = \psi(x) . \end{array} \right . 
\]
In particular, we have
$$\|\nabla u(t)\|_{L^2}\sim \frac{1}{|t|} \ \ \mbox{as} \ \ t\nearrow 0.$$
\end{theorem}

\noindent {\bf Comments on the Result}\\

1. {\it Long range issue}: Our result is in the spirit of Bourgain-Wang \cite{Bourgain+Wang1997}  who treated the case of local nonlinearities in $d=1,2$ space dimensions. However, due to the nonlocal nature of the Hartree nonlinearity which is long range in some sense, our proof departs in some respect from Bourgain-Wang method by introducing some modulation theory and by exploiting radial symmetry to decouple the blowup part from radiation. As sketched below, we expect our result to be generalizable to nonradial data, provided that some implicit conditions are imposed on $\psi(x)$; see the remark at the end of Section \ref{sec:proof:BW}.\\

2. {\it  Scattering}: The Bourgain-Wang strategy is based on the construction of some nonlinear wave operator. Undoing the pseudo conformal transformation, the statement is equivalent to proving some finite codimensional stability of $Q$, i\,e.~we exhibit {\it global} solutions with 
$$u(t,x)=Q(x)e^{it}+\tilde{z}_{\phi}(t,x)+\e(t,x) .$$
Here $\tilde{z}_{\phi}$ is the scattering wave and $\|\e\|_{H^{1}}\to 0$ as $t\to+\infty$. This strategy is very robust and we expect that it would carry over to the case $d=3$ and $\Phi(x)=\frac{1}{|x|}$ to construct non-trivial solutions that disperse to $Q$, which would extend the results in \cite{FTY2002}. See also \cite{Cote2007}, \cite{Krieger+Schlag2006} for related results with local nonlinearities.\\

Let us conclude by saying that both Theorem \ref{th:main} and Theorem \ref{th:BW} rely on solving in some sense the Cauchy problem from infinity. The strength of this strategy it that is does not require fine dispersive estimates on the propagator of the linearized flow close to the ground state. One should think that the long range structure of the problem actually make this last question quite delicate. However, all we need is to ensure an at most polynomial instability of the flow close to $Q$, which relies on elliptic nondegeneracy properties of the linearized operator. As initiated by Weinstein \cite{Weinstein1985}, such properties rely on the variational characterization of the ground state and a nondegeneracy result for the linearized operator. For the Hartree equation considered here, the nondegeneracy of the linearized operator does not follow from an adaptation of Weinstein's argument. Rather, our nondegeneracy proof will be based on an argument given by Lenzmann \cite{Lenzmann2008} for a Hartree NLS in dimension $d=3$; see Theorem \ref{th:ker} and Section \ref{sec:spectral} below.

\subsection*{Outline and Notation}
Theorems \ref{th:main} and \ref{th:BW} will be proven in Sections \ref{sec:proof:main} and \ref{sec:proof:BW}, respectively. In Section \ref{sec:spectral}, we prove uniqueness of ground states $Q$ and, as a main technical result, the nondegeneracy of the linearized operator $L_+$ close to $Q$, see Theorem \ref{th:ker} below. In Section \ref{sec:Qt}, we construct a modified class of ground state-like profiles called $\Qt$.

In what follows, we shall employ standard notation. By $a \lesssim b$ we mean that $a \leq C b$ for some positive constant $C >0$, which is allowed to depend on $k$ appearing in Theorem \ref{th:main}, as well as some large constant $T_0 > 0$ to be chosen in Section \ref{sec:proof:main} below. We remind the reader that we work in $d = 4$ space dimensions throughout the rest of this paper.

\subsection*{Acknowledgments}
J.~K.~is partially supported by National Science Foundation Grant DMS-0757278 and a Sloan Foundation Fellowship. E.~L.~gratefully acknowledges partial support by the National Science Foundation Grant DMS-0702492. P.R. was supported by the Agence Nationale de la Recherche, ANR Projet Blanc OndeNonLin and ANR jeune chercheur SWAP.

\section{Existence of critical mass blowup solutions}\
 \label{sec:proof:main}

This section is devoted to the proof Theorem \ref{th:main}. We shall freely use some results whose proofs are postponed to Sections \ref{sec:spectral} and \ref{sec:Qt} below.


\subsection{Reformulation of the problem}


Let us start with the following observation. Let $k > 0$ be a fixed number and assume that $v = v(t,x)$ is a sufficiently smooth radial solution of
\begin{equation} \label{eq:vPDE}
i \pr_t v = -\Delta v -\left ( \frac{\phi(t^{-k} |\cdot|^k)}{|\cdot|^2} \ast |v|^2 \right ) v ,
\end{equation}
for times $t \geq T_0$, where $T_0>0$ a large constant. An elementary calculation shows that
\begin{equation} \label{eq:ufromv}
u(t,x) = \frac{1}{t^2} e^{\frac{i x^2}{4t}} v \big (\frac{-1}{t}, \frac{x}{t} \big )
\end{equation}
solves (\ref{eq:hartree}) on the time interval $[-T_0^{-1}, 0)$. Our goal is now to construct a global solution $v(t,x)$ to (\ref{eq:ufromv}) such that: 
$$v(t,x)-Q(x)e^{it}\to 0  \ \ \mbox{in} \ \ \Sigma \ \ \mbox{as} \ \ t\to +\infty$$ where $Q$ is the ground state solution to (\ref{eqsoliton}). If we introduce a decomposition 
\be
\label{chehio}
v(t,x)=e^{it}[Q(x)+\e(t,x)]
\ee and try to solve for $\e$,  we will have in the RHS of the $\e$ equation a term like:
\be
\label{estpourri}
\big (\frac{\phi(t^{-k}|\cdot|^{k})}{|\cdot|^{2}}\ast |Q|^{2} \big ) \epsilon - \big ( \frac{1}{|\cdot|^{2}}\ast |Q|^{2} \big )  \epsilon =O \big (\frac{\epsilon}{t^{2}} \big )
\ee
if only $H^1$ control on $\e$ is known. This information is not sufficient to counteract losses due to the algebraic degeneracy of the generalized null-space of the linear operator close to $Q$. For this reason, our first step is to introduce a modified ground state profile called $\Qt$. To this end, we let $\Qinf \in H^1(\R^4)$ denote the ground state solution to 
\begin{equation}\label{Q_infty}
-\DD \Qinf - \big ( \frac{1}{|x|^2} \ast |\Qinf|^2 \big ) \Qinf = -\Qinf, \quad \Qinf(r) > 0 .
\end{equation}
Then the next result follows from an implicit function argument and the nondegeneracy of the linearized operator at $Q^{\infty}$. The proof of this result is postponed to Sections \ref{sec:spectral} and \ref{sec:Qt} below.

\begin{Prop}\label{thm:Q} 
Let $k >0$ be fixed in (\ref{Q_t}) and suppose $\phi(\cdot)$ satisfies the assumptions in Theorem \ref{th:main}. Then, for $T_0 =T_0(k) > 0$ sufficiently large, the following properties hold.
\begin{enumerate}
\item[(i)] There exists a family $\{ \Qt \}_{t \geq T_0}$ of radial, real-valued solutions to 
\be
\label{Q_t}
-\Delta Q^{(t)}-\left(\frac{\phi(t^{-k}|\cdot|^k)}{|\cdot|^2}\ast |Q^{(t)}|^2\right)Q^{(t)}=-Q^{(t)}
\ee
such that $t \mapsto \Qt \in H^1(\R^4)$ is $C^1$ and 
\[ \lim_{t \to \infty} \| \Qt - \Qinf \|_{H^1_x} = 0 . \]

\item[(ii)] We have the following bound
\be
\label{boundpat}
\| \pr_t \Qt \|_{H^1_x} \lesssim t^{-k-1} , \quad \mbox{for $t \in [T_0, \infty)$}.
\ee
\end{enumerate}
\end{Prop}

\begin{Rk} The bound (\ref{boundpat}) means that for the $Q^{(t)}-Q^{\infty}$ part of $\e$ in the decomposition (\ref{chehio}), the estimate (\ref{estpourri}) can be improved to gain $O(\frac{1}{t^k})$. This is very much a consequence of the uniform exponential decay of $Q^{(t)}$, see (\ref{estg}).
\end{Rk}

We now aim at finding $v$ solution to (\ref{eq:vPDE}) and introduce a decomposition:
\be
\label{eq:ansatz}
v(t,x)=e^{it}[Q^{(t)}+\e(t,x)].
\ee
Then the equation for $\epsilon$, which we record in vectorial notation so that the linear operator is actually $\C$-linear, is the following 
\begin{equation}\label{eq:eps}
i\partial_{t}\bm\epsilon\\ \overline{\epsilon}\endm+\calH^{(t)}\bm\epsilon\\ \overline{\epsilon}\endm= \bm F(t,x,\epsilon)\\ -\overline{F(t,x,\epsilon)}\endm .
\end{equation}
Here $\Ht$ is found to be matrix-valued (non self-adjoint) operator
\begin{equation} \label{eq:Hinf}
\Ht = \left ( \begin{array}{cc} \DD - 1 + \Vt + \Wt & \Wt \\ -\Wt & -\DD + 1 - \Vt - \Wt \end{array} \right ),
\end{equation}
where $\Vt$ and $\Wt$ are bounded operators on $L^2(\RR^4)$ (as one easily verifies) which are given by
\begin{equation} \label{def:VtWt}
\Vt \xi =  \big ( \frac{\phi(t^{-k} |\cdot|^k)}{|\cdot|^2} \ast |\Qt|^2 \big ) \xi, \quad \Wt \xi = \Qt \big ( \frac{\phi(t^{-k} |\cdot|^k )}{|\cdot|^2} \ast (\Qt \xi ) \big ) . 
\end{equation}
Note that $\Wt$ is a nonlocal operator. Furthermore, the forcing term $F$ in (\ref{eq:eps}) reads
\begin{align}
F(t,x,\epsilon) & = - \left \{ i \pr_t \Qt  + \left ( \frac{\phi(t^{-k} |\cdot|^k)}{|\cdot |^2} \ast (\Qt (\epsilon + \overline{\epsilon} ) ) \right ) \epsilon \right . \nonumber \\
& \qquad \qquad \left . + \left ( \frac{\phi(t^{-k} |\cdot|^k)}{|\cdot |^2} \ast |\epsilon|^2 \right ) (\Qt + \epsilon ) \right \} . \label{eq:F}
\end{align}
Note that we have the regularity $F \in H^1(\RR^4)$, as can be verified using the Hardy-Littlewood-Sobolev inequality etc.\\

Theorem \ref{th:main} is now a consequence of the following:

\begin{prop}[Solutions to the $\e$ equation]
\label{prop:epsilon} For $k \geq 5$ and $T_0 =T_0(k) > 0$ sufficiently large, the equation \eqref{eq:eps} admits a solution $\epsilon\in C^{0}_t H^1_x([T_{0},\infty) \times \RR^4)$ with the additional property 
\be
\label{oo}
|| x \epsilon(t)||_{L_{x}^{2}}\leq \delta t^{-k+5}, \ \ \|\e(t)\|_{H^1_x}\lesssim t^{-k+4}\quad \mbox{for $t \geq T_0$}, 
\ee
where $\delta>0$ can be chosen arbitrarily small, provided that $T_{0}$ is sufficiently large. 
\end{prop}

\begin{Rk} Observe that the critical mass condition $\|u_0\|_{L^2}=\|Q\|_{L^2}$ is the consequence of the strong convergence (\ref{oo}) and the conservation of the $L^2$ norm.
\end{Rk}

\subsection{Spectral structure of $\Ht$}


The proof of Proposition \ref{prop:epsilon} relies on the {\it algebraic} instability of the linearized operator close to $Q^{\infty}.$ The following proposition is a standard consequence of the variational characterization of $Q$ and some nondenegeracy properties, see Theorem \ref{th:ker} below.

\begin{lemma}[Spectra structure of $\Hinf$]
\label{lem:H_inf} The operator $\calH^{(\infty)}$ acting on $L^{2}_{\text{rad}}(\R^{4}, \C^{2})$ has the following properties.
\begin{enumerate}

\item[(i)] The essential spectrum is $\sigma_{\rm ess}(\Hinf) = (-\infty, -1] \cup [1,\infty)$.
\item[(ii)] The generalized null-space  
\[ \calN = \big \{ f \in L^{2}_{\text{rad}}(\R^{4}, \C^{2}) : \mbox{$\exists m \in \mathbb{N}$ such that $(\Hinf)^m f = 0$} \big \} \] 
has dimension $\dim \mathcal{N} = 4$ and is generated by the following functions: 
\[
\phi_{1} =\bm i Q^{(\infty)}\\ -iQ^{(\infty)}\endm,\quad \phi_{2} =\bm 2Q^{(\infty)}+x\cdot \nabla Q^{(\infty)}\\ 2Q^{(\infty)}+x\cdot \nabla Q^{(\infty)}\endm, 
\]
\[
\phi_{3} =\bm i|x|^{2}Q^{(\infty)}\\ -i|x|^{2}Q^{(\infty)}\endm,\quad \phi_{4}=\bm \rho \\ \rho \endm ,
\]
where $\rho \in L^2_{\mathrm{rad}}(\RR^4)$ is the unique solution of
\[
L_{+}\rho=-|x|^{2} \Qinf ,
\]
with 
\[
L_{+}=-\DD + 1 - \big ( \frac{1}{|x|^2} \ast |\Qinf|^2 \big ) - 2 \Qinf \big ( \frac{1}{|x|^2} \ast (\Qinf \cdot) \big ).
\]
Moreover, the function $\rho$ is radial, smooth and exponentially decaying. 

\item[(iii)] We have following bound for the linear evolution associated to $\calH^{(\infty)}$:
\[
\big \| e^{it\calH^{(\infty)}}\bm f\\ \overline{f}\endm \big \|_{L^{2}_x}\lesssim\big (1+|t|^{3} \big )||f||_{L_{x}^{2}} .
\]
\end{enumerate}
\end{lemma}

\begin{remarks*} {\em 
1) The proof of Lemma \ref{lem:H_inf} (given in Section \ref{sec:spectral} below) relies on a careful analysis of $\mathrm{ker} \, L_{+}$, and it otherwise mirrors earlier work by Weinstein \cite{Weinstein1985} for local NLS.

2) Due to the long-range behavior of the potential term $\Vinf$, we expect $\Hinf$ to have infinitely many non-zero eigenvalues in $(-1,1)$. Such eigenvalues play no role in our analysis, though.

3) Since $\Hinf$ is not self-adjoint, one has to be careful about what is meant by $\sigma_{\mathrm{ess}}(\Hinf)$. However, by adapting the arguments in \cite{Hundertmark+Lee2007} and using the special matrix structure of $\Hinf$, we see that $\sigma_{\mathrm{ess}}(\Hinf) = \sigma(\Hinf) \setminus \sigma_{\mathrm{disc}}(\Hinf)$ which is a well-known fact for self-adjoint operators. Recall that the discrete spectrum $\sigma_{\mathrm{disc}}(\Hinf)$ is the set of all isolated $\lambda \in \sigma(\Hinf)$ with finite algebraic mutliplicity.  

 }
\end{remarks*}

Since $\Qt \to \Qinf$ in $H^1$ as $t \to \infty$, standard perturbation theory allows us to conclude the following for $\Ht$:

\begin{lemma}\label{lem:H_t} Let $k >0$ and choose $T_0 = T_0(k) > 0$ sufficiently large. Then the following properties hold.
\begin{enumerate}
\item[(i)] For any $t \in [T_0,\infty)$, the essential spectrum is $\sigma_{\mathrm{ess}}(\Ht) = (-\infty, -1] \cup [1,\infty)$.
\item[(ii)] There exists $c > 0$ such that the projections
\[
\Prt = \frac{1}{2 \pi i} \oint_{|z| = c} (z - \Ht)^{-1} \, dz
\]
exist for all $t \in [T_0,\infty)$. Moreover, we have 
\[ \lim_{t \to \infty} \| \Prt - \Prinf \|_{H^1_x \to H^1_x} = 0, \]
where $\Prinf$ denotes the projection onto the generalized null-space $\mathcal{N}$ of $\Hinf$.
\end{enumerate}
\end{lemma}

\begin{remark*} {\em
The proof of Lemma \ref{lem:H_t} is also relegated to Section \ref{sec:spectral} below.
  }
\end{remark*}


\subsection{Setting up the Iteration Scheme for $\epsilon$}


Let us now turn to the construction of $\epsilon$, as claimed in Proposition \ref{prop:epsilon}.\\
To this end, we decompose the source term $F(t,x,\epsilon)$ in (\ref{eq:eps}) into a root part as well as a ``non-root''\footnote{We use this somewhat awkward terminology rather than the more customary ``dispersive" on account of the presence of real eigenvalues in the spectral gap which prevent dispersive behavior of the linear evolution, even when projecting away the root modes.} part, both with respect to $\Hinf$. That is, 
\begin{equation}
F(t,x,\epsilon)=\sum_{i=1}^{4}b_{i}(t)\phi_{i}(x)+ \Pinf F(t,x,\epsilon),
\end{equation}
where $\{ \phi_i \}_{i=1}^4$ span the generalized null-space of $\Hinf$ (see Lemma \ref{lem:H_inf}), and $\Pinf$ is given by
\begin{equation}
\Pinf = 1 - \Prinf .
\end{equation}
Here $\Prinf$ denotes (as before) the projection onto the generalized null-space of $\Hinf$. Clearly, we have 
\begin{equation} \label{eq:modul}
\la \Pinf F(t,x,\epsilon), \psi_{i}(x)\ra =0,\quad \mbox{for $i=1,2,3,4$,}
\end{equation}
Let us reformulate the latter statement in terms of the dual root modes $\{ \psi_{i} \}_{i=1}^4$ which generate the generalized null-space associated with the adjoint operator $(\calH^{(\infty)})^{*}$. These modes are given by 
\begin{equation}
\psi_{1} =\bm Q^{(\infty)}\\ Q^{(\infty)}\endm,\,\psi_{2} =\bm i(2Q^{(\infty)}+x\cdot \nabla Q^{(\infty)}) \\ -i(2Q^{(\infty)}+x\cdot \nabla Q^{(\infty)})\endm, 
\end{equation}
\begin{equation}
 \psi_{3} =\bm |x|^{2}Q^{(\infty)}\\ |x|^{2}Q^{(\infty)}\endm,\,\psi_{4} =\bm i\rho \\ -i\rho \endm,
\end{equation}
with $\rho$ from Lemma \ref{lem:H_inf}. Then \eqref{eq:modul} reads
\begin{equation}\label{root1}
2b_{4}\la \rho, Q^{(\infty)}\ra =\la F(t,x,\epsilon), \psi_{1}(x)\ra 
\end{equation}
\begin{equation}\label{root2}
2b_{2}\la |x|^{2}Q^{(\infty)}, Q^{(\infty)}\ra =\la F(t,x,\epsilon), \psi_{2}(x)\ra 
\end{equation}
\begin{equation}\label{root3}
-2b_{3}\la |x|^{2}Q^{(\infty)}, Q^{(\infty)}\ra +2b_{4}\la |x|^{2}Q^{(\infty)}, \rho\ra = \la F(t,x,\epsilon), \psi_{3}(x)\ra 
\end{equation}
\begin{equation}\label{root4}
-2b_{1}\la \rho, Q^{(\infty)}\ra -2b_{3}\la |x|^{2}Q^{(\infty)}, \rho\ra =\la F(t,x,\epsilon), \psi_{4}(x)\ra 
\end{equation}
A calculation shows the essential fact that the coefficients of the $b_{i}$ are all positive numbers. Note that all the numbers on the right are purely imaginary. Also, this linear system for the $b_{i}$ is non-singular, so that we can solve for each $b_{i}$ when the right-hand side is given. \\

Now we set up an iteration scheme to solve \eqref{eq:eps}, where the zeroth iterate is $\epsilon_{0}(t,x)=0$.  Assume now we have constructed the $\epsilon_q$; then we define the next iterate $\epsilon_{q+1}$ via a nested iteration procedure as follows: As before, decompose the source term as
\begin{equation}
F(t,x,\epsilon_{q})=\sum_{i=1}^{4}b_{i,q}(t)\phi_{i}(x) + \Pinf F(t,x,\epsilon_{q}) .
\end{equation}
Then we need to solve 
\begin{equation}
\label{eps}
i\partial_{t}\bm\epsilon_{q+1}\\ \overline{\epsilon_{q+1}}\endm+\calH^{(t)}\bm\epsilon_{q+1}\\ \overline{\epsilon_{q+1}}\endm=\sum_{i=1}^{4}b_{i,q}(t)\phi_{i}(x)+ \Pinf F(t,x,\epsilon_{q}) ,
\end{equation}
which can be done via a sequence of approximate solutions as follows. First, define $\epsilon_{q+1}^{1}$ as a sum of terms according to
\begin{equation}
\bm\epsilon_{q+1}^{1}\\ \overline{\epsilon_{q+1}^{1}}\endm=\sum_{i=1}^{4}a_{i,q}^{1}(t) \phi_{i}(x) +\bm\tilde{\epsilon}^{1}_{q+1}\\ \overline{\tilde{\epsilon}^{1}_{q+1}}\endm  .
\end{equation}
Here we define the {\bf Root Part}
\begin{equation}
 \sum_{i=1}^{4}a_{i,q}^{1}(t) \phi_{i}(x) 
\end{equation}
and the {\bf Non-root Part} 
\begin{equation}
\bm\tilde{\epsilon}^{1}_{q+1}\\ \overline{\tilde{\epsilon}^{1}_{q+1}}\endm 
\end{equation}
as follows.

\subsection*{Definition of Root Part} Let $\{ a_{i,q}^1(t)\}_{i=1}^4$ be the solutions vanishing at infinity of the following coupled system of ODE's (i.\,e.~the modulation equations): 
\begin{equation}\label{modulation1}
\dot{a}_{1,q}^{1}-2a_{2,q}^{1}=\frac{b_{1,q}}{i} , \quad \dot{a}_{2,q}-4a_{3,q}=\frac{b_{2,q}}{i} ,
\end{equation}
\begin{equation}\label{modulation2}
\dot{a}_{3,q}+a_{4,k}=\frac{b_{3,q}}{i}, \quad  \dot{a}_{4,q}=\frac{b_{4,q}}{i} .
\end{equation}
This choice is easily seen to imply that
\begin{align*}
(i\partial_{t}+\calH^{(t)})\sum_{i=1}^{4}a_{i,q}^{1}(t) \phi_{i}(x) & =(\calH^{(t)}-\calH^{(\infty)})\sum_{i=1}^{4}a_{i,q}^{1}(t) \phi_{i}(x) + \sum_{i=1}^{4}b_{i,q}(t)\phi_{i}(x) .
\end{align*}

\subsection*{Definition of Non-root Part} Let $\Prt$ be the projections given by Lemma \ref{lem:H_t} and put $\Pt = 1 -  \Prt$. Next, we define $\bm\tilde{\epsilon}^{1}_{q+1}\\ \overline{\tilde{\epsilon}^{1}_{q+1}}\endm $ as the solution of the linear inhomogeneous problem 
\begin{equation}\label{auxiliary}
(i\partial_{s}+\calH^{(t)})\bm\tilde{\epsilon}^{1,(t)}_{q+1}(s)\\ \overline{\tilde{\epsilon}^{1,(t)}_{q+1}}(s)\endm 
= \Pt F(s,x,\epsilon_{q}) , \quad \mbox{for $s \geq t$},
\end{equation}
such that $\tilde{\epsilon}^{1,(t)}_{q+1}(s) \to 0$ as $s \to +\infty$, evaluated at time $s=t$. That is, we set
\begin{equation}
\bm\tilde{\epsilon}^{1}_{q+1}\\ \overline{\tilde{\epsilon}^{1}_{q+1}}\endm =\bm\tilde{\epsilon}^{1,(t)}_{q+1}(t)\\ \overline{\tilde{\epsilon}^{1,(t)}_{q+1}}(t)\endm  .
\end{equation}
That $\tilde{\epsilon}^{(1,t)}_{q+1}(s)$ indeed exists will follow from the proof of Proposition \ref{prop:subiteration} below. It is important to note that we treat the variable $t$ in \eqref{auxiliary} as a fixed parameter, while the time variable is denoted by $s$. Furthermore, we note
\begin{equation}
(i\partial_{t}+\calH^{(t)})\bm\tilde{\epsilon}^{1}_{q+1}\\ \overline{\tilde{\epsilon}^{1}_{q+1}}\endm = \Pt F(s,x,\epsilon_{q})+i\partial_{t}\bm\tilde{\epsilon}^{1,(t)}_{q+1}(s)\\  \overline{\tilde{\epsilon}^{1,(t)}_{q+1}}(s)\endm \Big |_{s=t}
\end{equation}
Combining now the definitions of the root and non-root part,  we deduce that 
 \begin{equation}\nonumber\begin{split}
 (i\partial_{t}+\calH^{(t)})\bm\epsilon_{q+1}^{1}\\ \overline{\epsilon_{q+1}^{1}}\endm
 & =F(t,x,\epsilon_{q})+(P^{(t)}-P^{\infty})F(t,x,\epsilon_{q})+i\partial_{t}\bm\tilde{\epsilon}^{1,(t)}_{q+1}(s)\\ \overline{\tilde{\epsilon}^{1,(t)}_{q+1}}(s)\endm \Big |_{s=t}\\&\quad +(\calH^{(t)}-\calH^{(\infty)})\sum_{i=1}^{4}a_{i,q}^{1}(t) \phi_{i}(x)\\
 & =:F(t,x,\epsilon_{q})+\text{error}_{q}^{1} .
 \end{split}\end{equation}
 Then the higher iterates $\bm\epsilon_{q+1}^{l}\\ \overline{\epsilon_{q+1}^{l}}\endm$, for $l\geq 2$, are defined inductively as follows: 
 \begin{equation}\label{eps_k^l}
  (i\partial_{t}+\calH^{(t)}) \left [\bm\epsilon_{q+1}^{l+1}\\ \overline{\epsilon_{q+1}^{l+1}}\endm- \bm\epsilon_{q+1}^{l}\\ \overline{\epsilon_{q+1}^{l}}\endm \right ]=\text{error}_{q}^{l}+\text{error}_{q}^{l+1} .
  \end{equation}
This completes the definition of our iteration scheme. 
  
\begin{remark*} {\em 
The term $\bm\epsilon_{q+1}^{l+1}\\ \overline{\epsilon_{q+1}^{l+1}}\endm- \bm\epsilon_{q+1}^{l}\\ \overline{\epsilon_{q+1}^{l}}\endm$ is constructed from $\text{error}_{q}^{l}$ just like $\bm\epsilon_{q+1}^{1}\\ \overline{\epsilon_{q+1}^{1}}\endm$ was constructed from $F(t,x,\epsilon_{q})$. }
\end{remark*}


\subsection{Construction of  $\epsilon_{q+1}$ under a bootstrap assumption}


We now construct $\epsilon_{q+1}$ form $\epsilon_{q}$ under a bootstrapping assumption.

\begin{prop}\label{prop:iteration}
Let $k \geq 5$ and choose $T_0 > 0$ sufficiently large. If 
\be
\label{estboot}
||F(t,x,\epsilon_{q})||_{H_{x}^{1}}\lesssim t^{-k},\quad \mbox{for $t \geq T_0$},
\ee
then equation \eqref{eps} admits a solution $\bm \epsilon_{q+1}\\ \overline{\epsilon_{q+1}}\endm$ satisfying 
\[
\big \| \bm \epsilon_{q+1}\\ \overline{\epsilon_{q+1}}\endm \big \|_{H_{x}^{1}}\lesssim t^{-k+4} , \quad \mbox{for $t \geq T_0$}.
\]
\end{prop}

Proposition \ref{prop:iteration} is a direct consequence of the following Lemma used iteratively in $l$ which allows to construct $\bm \epsilon_{q+1}\\ \overline{\epsilon_{q+1}}\endm$ as the limit of the sequence of iterates $\bm \epsilon^{l}_{q+1}\\ \overline{\epsilon^{l}_{q+1}}\endm$.\\

\begin{lemma}
\label{prop:subiteration} Using the notation from above, assume that $||F(t,x,\epsilon_{q})||_{H_{x}^{1}}\lesssim t^{-k}$ for $t\geq T_{0}$. Furthermore, suppose $k\geq 5$ and let $T_{0} >0$ be sufficiently large. Then
  \begin{equation}
  ||\epsilon_{q+1}^{1}(t,x)||_{H_{x}^{1}}\lesssim t^{-k+4},\quad ||\mathrm{error}_{q}^{1}||_{H_{x}^{1}}\lesssim \delta t^{-k}, \quad \mbox{for $t \geq T_0$.}
  \end{equation}
  Here the implied constants are universal (in particular do not depend on $l$), and $\delta>0$ can be chosen arbitrarily small, provided that $T_0 > 0$ is chosen sufficiently large. In particular, the corrections applied to the $\epsilon_{q}^{1}(t,x)$ decay exponentially in $l$, whence these functions converge in the $H^{1}$-topology. 
  \end{lemma}
  
  \begin{proof}  We now prove Lemma \ref{prop:subiteration}. We first show the bound for $\epsilon_{q+1}^{1}(t,x)$, which is split into a root and a non-root part. For the root part, the system \eqref{modulation1}-\eqref{modulation2} as well as \eqref{root1}-\eqref{root4} easily imply that 
  \begin{equation}
  \sum_{i=1}^{4}|a_{i,q}^{1}(t)|\lesssim t^{-k+4} .
  \end{equation}
  Furthermore, since the root modes $\{ \phi_{i}(x) \}_{i=1}^4$ are of exponential decay, we infer
  \begin{equation}
  |(\calH^{(t)}-\calH^{(\infty)})\sum_{i=1}^{4}a_{i,q}^{1}(t)\bm \phi_{i}(x)\\ \overline{\phi_{i}(x)}\endm|\lesssim t^{-k} t^{-k+4}\leq \delta t^{-k},
  \end{equation}
  provided $T_{0}$ is chosen large enough and $t\geq T_{0}$. Next, we consider $\tilde{\epsilon}^{1,(t)}_{q+1}(s)$. In view of \eqref{auxiliary}, we have 
  \begin{equation}
  \bm\tilde{\epsilon}^{1,(t)}_{q+1}(s)\\ \overline{\tilde{\epsilon}^{1,(t)}_{q+1}}(s)\endm 
=\Pt \bm\tilde{\epsilon}^{1,(t)}_{q+1}(s)\\ \overline{\tilde{\epsilon}^{1,(t)}_{q+1}}(s)\endm 
\end{equation}
We shall also use the notation 
\begin{equation}
\Pt \bm\tilde{\epsilon}^{1,(t)}_{q+1}(s)\\ \overline{\tilde{\epsilon}^{1,(t)}_{q+1}}(s)\endm
=:\bm \Pt \tilde{\epsilon}^{1,(t)}_{q+1}(s)\\ \overline{\Pt \tilde{\epsilon}^{1,(t)}_{q+1}}(s)\endm
\end{equation}
and similarly for $\Pinf$. Next, we claim the following estimate to be true: 
\begin{equation} \label{ineq:claim}
|| \bm\tilde{\epsilon}^{1,(t)}_{q+1}(s)\\ \overline{\tilde{\epsilon}^{1,(t)}_{q+1}}(s)\endm||_{H_{x}^{1}}\lesssim ||F(s',x,\epsilon_{q})||_{L_{s'}^{1}H_{x}^{1}([s,\infty)\times\R^{4})} .
\end{equation}
Indeed, recall the definition of $\Hinf$ from (\ref{eq:Hinf}) and let (as before)
\begin{align}
L_{+}  = - \DD + 1 -\big ( \frac{1}{|x|^{2}} \ast |\Qinf|^2 \big )  - 2 \Qinf \big ( \frac{1}{|x|^{2}} \ast (Q^{(\infty)} \cdot ) \big ) ,
\end{align}
and 
\begin{equation}
 L_{-}= -\Delta + 1 - \big ( \frac{1}{|x|^{2}} \ast |\Qinf|^2 \big ).
\end{equation}
 Since $\mathrm{ker} \, L_+ = \{ 0\}$ in the radial sector, by Theorem \ref{th:ker} below, an adaptation of a well-known argument by Weinstein \cite{Weinstein1985} for NLS with local nonlinearities yields the coercivity estimate (in the radial sector):
 \begin{equation}
 \la L_{-} \Pinf \tilde{\epsilon}_{q+1}^{1,(t)},  \Pinf \tilde{\epsilon}_{q+1}^{1,(t)} \ra+\la L_{+} \Pinf \tilde{\epsilon}_{q+1}^{1,(t)},  \Pinf \tilde{\epsilon}_{q+1}^{1,(t)}\ra \gtrsim \| \Pinf \tilde{\epsilon}_{q+1}^{1,(t)} \|_{H_{x}^{1}}^{2} .
 \end{equation}
 Furthermore, by the continuity properties stated in Theorem \ref{thm:Q} and Lemma \ref{lem:H_t},
  \begin{equation} \begin{split}
  &\la L_{-}^{(t)} \Pt \tilde{\epsilon}_{q+1}^{1,(t)},  \Pt \tilde{\epsilon}_{q+1}^{1,(t)}\ra+\la L_{+}^{(t)} \Pt \tilde{\epsilon}_{q+1}^{1,(t)},  \Pt \tilde{\epsilon}_{q+1}^{1,(t)}\ra \nonumber \\
  &= \la L_{-} \Pinf \tilde{\epsilon}_{q+1}^{1,(t)},  \Pinf \tilde{\epsilon}_{q+1}^{1,(t)}\ra+\la L_{+} \Pinf \tilde{\epsilon}_{q+1}^{1,(t)},  \Pinf \tilde{\epsilon}_{q+1}^{1,(t)}\ra+o(\|\tilde{\epsilon}_{q+1}^{1,(t)}\|_{H^1_x}),
  \end{split}\end{equation}
  where in the last line the $o(..)$ means that this quantity vanishes as $t\rightarrow \infty$, and we use the notation
  \begin{equation}
 L_{+}^{(t)}=  -\DD + 1 - \big ( \frac{\phi(t^{-k} |\cdot|^k)}{|\cdot|^2} \ast |\Qt|^2 \big ) - 2 \Qt \big ( \frac{\phi(t^{-k} |\cdot|^k)}{|\cdot|^2} \ast (\Qt \cdot ) \big ),
 \end{equation}
 \begin{equation}
 L_{-}^{(t)}= -\DD + 1 - \big ( \frac{\phi(t^{-k} |\cdot|^k)}{|\cdot|^2} \ast |\Qt|^2 \big ).
  \end{equation}
 Finally, we note that
 \begin{equation}
 \Pt \bm\tilde{\epsilon}^{1,(t)}_{q+1}(s)\\ \overline{\tilde{\epsilon}^{1,(t)}_{q+1}}(s)\endm 
 =\int_{s}^{\infty}e^{i(s-s')\calH^{(t)}}\Pt F(s',x,\epsilon_{q}) \, ds',
 \end{equation}
and the quadratic form $\la L_{-}^{(t)}.,.\ra +\la L_{+}^{(t)}.,.\ra$ is invariant under the evolution associated with $\calH^{(t)}$, whence the claimed estimate (\ref{ineq:claim}) follows. By assumption on the forcing term, we thus have shown that 
\begin{equation}
\big \|  \bm\tilde{\epsilon}^{1,(t)}_{q+1}(t)\\ \overline{\tilde{\epsilon}^{1,(t)}_{q+1}}(t)\endm \big \|_{H_{x}^{1}}\lesssim t^{-k+1} ,
\end{equation}
whence the first estimate of Lemma \ref{prop:subiteration} follows. 

\medskip
Next, consider the error due to $\bm\tilde{\epsilon}^{1,(t)}_{q+1}(t)\\ \overline{\tilde{\epsilon}^{1,(t)}_{q+1}}(t)\endm $, which equals
\begin{equation}
[\Pt-\Pinf]F(t,x,\epsilon_{q})+i\partial_{t}\bm\tilde{\epsilon}^{1,(t)}_{q+1}(s)\\ \overline{\tilde{\epsilon}^{1,(t)}_{q+1}}(s)\endm \Big |_{s=t}
\end{equation}
By Lemma \ref{lem:H_t}, we deduce
\begin{equation}
\| [\Pt-\Pinf ]F(t,x,\epsilon_{q})\|_{H^1_x} \lesssim \delta t^{-k},
\end{equation}
for $t \geq T_0$, provided that $T_{0}$ is large enough. Finally, we need to estimate the time derivative. Note that 
 \begin{equation*}
 \partial_{t}\bm\tilde{\epsilon}^{1,(t)}_{q+1}(s)\\ \overline{\tilde{\epsilon}^{1,(t)}_{q+1}}(s)\endm \Big |_{s=t}
=\Pt \partial_{t}\bm\tilde{\epsilon}^{1,(t)}_{q+1}(s)\\ \overline{\tilde{\epsilon}^{1,(t)}_{q+1}}(s)\endm \Big |_{s=t}
+(\partial_{t}\Pt)\bm\tilde{\epsilon}^{1,(t)}_{q+1}(s)\\ \overline{\tilde{\epsilon}^{1,(t)}_{q+1}}(s)\endm \Big |_{s=t}
\end{equation*}
To handle the second term, we use the following estimates. 
\begin{lemma} \label{lem:opbound}
For $k \geq 2$ and $T_0 > 0$ sufficiently large, we have
\begin{equation*}
||(\partial_{t} \Pt)||_{H^{1}_x \rightarrow H^{1}_x}\lesssim t^{-3},\quad ||(\partial_{t}\calH^{(t)})||_{H^{1}_x \rightarrow H^{1}_x}\lesssim t^{-3}, \quad \mbox{for $t \geq T_0$}.
\end{equation*}
\end{lemma} 
\begin{proof} 
We start by proving the bound for $\pr_t \Ht$. Recall the definitions of $\Vt$ and $\Wt$ from \eqref{def:VtWt} and observe that 
\begin{equation}
\partial_{t}\calH^{(t)}= \left ( \begin{array}{cc} \pr_t \Vt + \pr_t \Wt & \pr_t \Wt \\ -\pr_t \Wt & -\pr_t \Vt - \pr_t \Wt \end{array} \right ) .
\end{equation}
For simplicity, let us consider $\pr_t \Vt$ here, and we remark that $\pr_t \Wt$ is estimated in a similar way. We find
\begin{align*}
\pr_t \Vt & =  -k t^{-k-1} \big ( \frac{\phi'(t^{-k} |\cdot|^k) |\cdot|^k}{|\cdot|^2} \ast |\Qt|^2 \big ) + 2 \big ( \frac{\phi(t^{-k} |\cdot|^k)}{|\cdot|^2} \ast (\Qt \pr_t \Qt) \big ) \nonumber \\
& =: I^{(t)}+ II^{(t)}.
\end{align*}
Next, we estimate $I^{(t)}$ as follows. Since $|x| |\phi'(x)| \lesssim 1$, it follows that
\begin{equation}
t^{-k-1} \frac{ |\phi'(t^{-k} |x-y|^k)| |x-y|^k}{|x-y|^2} \lesssim t^{-3} .
\end{equation}
Hence, we find
\begin{equation}
\| I^{(t)} \|_{L^\infty_x} \lesssim t^{-3} \| \Qt \|_{L^2_x}^2 \lesssim t^{-3},
\end{equation}
thanks to the uniform bound $\| \Qt \|_{H^1_x} \lesssim 1$ implied by Theorem \ref{th:Qt}. Similarly, we obtain
\begin{equation}
\| \nabla I^{(t)} \|_{L^\infty_x} \lesssim t^{-3} \| \nabla |\Qt|^2 \|_{L^1_x} \lesssim t^{-3} \| \Qt \|_{L^2_x} \| \| \nabla \Qt \|_{L^2_x} \lesssim t^{-3}.
\end{equation}
Thus we have the operator bound $\| I^{(t)} \|_{H^1_x \to H^1_x} \lesssim t^{-3}$ for all $t$ sufficiently large.

As for proving such a bound for $II^{(t)}$, we argue as follows. Using $|\phi(x)| \lesssim 1$ and the Schwarz inequality, we deduce
\begin{align*} 
\| II^{(t)} \|_{L^\infty_x} & \lesssim \sup_{z \in \RR^4} \int_{\RR^4} \frac{|\Qt(y)| |\pr_t \Qt(y)|}{|z-y|^2} \, dy \nonumber \\
& \lesssim \sup_{z \in \RR^4} \big ( \big \| |\cdot - z|^{-1} \Qt \big \|_{L^2_x} \big \| |\cdot - z|^{-1} \pr_t \Qt \big \|_{L^2_x} \big ) \nonumber \\
& \lesssim \| \nabla \Qt \|_{L^2_x} \| \nabla \pr_t \Qt \|_{L^2_x} \lesssim t^{-k-1} .
\end{align*}
where we also used $\| |\cdot - z|^2 f \|_{L^2_x} \lesssim \| \nabla f \|_{L^2_x}$ for any $z \in \RR^4$, which follows from Hardy's inequality and translational invariance. Also, in the last step, we used the estimate of Lemma \ref{lem:tdecay} below. Next, we derive the following estimate:
\begin{align*}
\| \nabla II^{(t)} \|_{L^4_x} & \lesssim \big \| \frac{1}{|x|^2} \big \|_{L^{(2,\infty)}_x} \big \| \nabla (\Qt \pr_t \Qt ) \big \|_{L^{4/3}_x} \nonumber \\
& \lesssim \| \Qt \|_{L^4_x} \| \nabla \pr_t \Qt \|_{L^2_x} + \| \nabla \Qt \|_{L^2_x} \| \pr_t \Qt \|_{L^4_x}  \lesssim \| \pr \Qt \|_{H^1_x} \lesssim t^{-k-1} .
\end{align*}
Here we used the weak Young inequality and Sobolev's embedding $\| f \|_{L^4_x} \lesssim \| \nabla f \|_{L^2_x}$ in $\RR^4$, as well as Lemma \ref{lem:tdecay} again. Since $k \geq 2$ by assumption, our estimates show that 
\begin{equation}
\| II^{(t)} f \|_{H^1_x} \lesssim \| II^{(t)} \|_{L^\infty_x} \| \nabla f \|_{L^2_x} + \| \nabla II^{(t)} \|_{L^{4/3}_x} \| f \|_{L^4_x} \lesssim t^{-3} \| f \|_{H^1_x} ,
\end{equation}
which completes the proof of the claim that, for all $t$ sufficiently large,
\begin{equation}
\| \pr_t \Vt \|_{H^1_x \to H^1_x} \lesssim t^{-3} .
\end{equation}
Again, we remark that an analogous estimate can be derived for the nonlocal operator $\pr_t \Wt$ in a similar way. This completes the proof of the second inequality stated in Lemma \ref{lem:opbound}.

It remains to show the first inequality. To this end, we recall that $\Pt = 1  - \Prt$ with $\Prt$ from Lemma \ref{lem:H_t}. This leads to
\begin{equation}
\partial_{t} \Pt =\frac{1}{2\pi i}\oint_{|z|=c}(z-\calH^{(t)})^{-1}(\partial_{t}\calH^{(t)})(z-\calH^{(t)})^{-1} \, dz,
\end{equation}
whence the first inequality of Lemma \ref{lem:opbound} follows from the second.  \end{proof}

Let us now conclude the proof of Lemma \ref{prop:subiteration}. Applying Lemma \ref{lem:opbound}, we see that 
\begin{equation}
||(\partial_{t} \Pt)\bm\tilde{\epsilon}^{1,(t)}_{q+1}(s)\\ \overline{\tilde{\epsilon}^{1,(t)}_{q+1}}(s)\endm|_{s=t}||_{H_{x}^{1}}
\lesssim t^{-2-k}\leq \delta t^{-k},
\end{equation}
provided $t$ is large enough. Finally, we need to estimate 
\begin{equation}
\Pt \partial_{t}\bm\tilde{\epsilon}^{1,(t)}_{q+1}(s)\\ \overline{\tilde{\epsilon}^{1,(t)}_{q+1}}(s)\endm|_{s=t} .
\end{equation}
Here we note the identity 
\begin{equation*}
(i\partial_{s}+\calH^{(t)})\Pt \partial_{t}\bm\tilde{\epsilon}^{1,(t)}_{q+1}(s)\\ \overline{\tilde{\epsilon}^{1,(t)}_{q+1}}(s)\endm
=  \Pt \big [\partial_{t}\Pt F(s,x,\epsilon_{q})-\partial_{t}\calH^{(t)}\bm\tilde{\epsilon}^{1,(t)}_{q+1}(s)\\ \overline{\tilde{\epsilon}^{1,(t)}_{q+1}}(s)\endm \big ] .
\end{equation*}
Applying Lemma \ref{lem:opbound} again, we conclude that 
\begin{equation*}
||\Pt \partial_{t}\bm\tilde{\epsilon}^{1,(t)}_{q+1}(s)\\ \overline{\tilde{\epsilon}^{1,(t)}_{q+1}}(s)\endm||_{H_{x}^{1}}
\lesssim t^{-3}\int_{s}^{\infty}[||F(s',x,\epsilon_{q})||_{H_{x}^{1}}+||\bm\tilde{\epsilon}^{1,(t)}_{q+1}(s')\\ \overline{\tilde{\epsilon}^{1,(t)}_{q+1}}(s')\endm||_{H_{x}^{1}}] \, ds' .
\end{equation*}
Putting $s=t$, we obtain the bound 
\begin{equation}
||\Pt\partial_{t}\bm\tilde{\epsilon}^{1,(t)}_{q+1}(s)\\ \overline{\tilde{\epsilon}^{1,(t)}_{q+1}}(s)\endm \big |_{s=t}||_{H_{x}^{1}}\lesssim t^{-k-1}\leq \delta t^{-k},
\end{equation}
provided that $t$ is large enough. This concludes the proof of Lemma \ref{prop:subiteration}. \end{proof}


\subsection{Control of the nonlinear term}


We now need to derive the bootstrap estimate (\ref{estboot}) by controling the nonlinear terms given by (\ref{eq:F}).

\begin{lemma}\label{lem:F} Assume $||\epsilon||_{H_{x}^{1}}\lesssim t^{-k+4}$ for $t\geq T_{0}$ with $T_{0} > 0$ sufficiently large. Additionally, suppose that $k \geq 8$. Let $F(t,x,\epsilon)$ be given by (\ref{eq:F}), then:
\[
||F(t,x,\epsilon)||_{H_{x}^{1}}\leq t^{-k-1} , \quad \mbox{for $t \geq T_0$}.
\]
\end{lemma}

\begin{proof}[Proof of Lemma \ref{lem:F}]  Consider, e.\,g., the term
\begin{equation}
A :=  \big ( \frac{\phi(t^{-k}|\cdot |^{k})}{|\cdot|^{2}} \ast |\epsilon|^{2} \big ) \epsilon .
\end{equation}
Using the Hardy's inequality $|x|^{-2} \lesssim (-\Delta)$ and H\"older's inequality, we conclude 
\begin{equation}
\| A \|_{L_{x}^{2}}\lesssim \| \epsilon \|_{\dot{H}^{1}_x}^{2}||\epsilon||_{L_{x}^{2}}\lesssim t^{-3k+12} \leq t^{-k-1},
\end{equation}
for $t$ sufficiently large. Next, we have 
\begin{equation}
\nabla A
= \big ( \frac{\phi(t^{-k}|\cdot |^{k})}{|\cdot|^{2}} \ast |\epsilon|^{2} \big ) \nabla \epsilon + \big ( \frac{\phi(t^{-k}|\cdot |^{k})}{|\cdot|^{2}} \ast (\nabla |\epsilon|^{2}) \big ) \epsilon .
\end{equation}
For the first term, we use Hardy's inequality again to conclude 
\begin{equation}
\big \| \big ( \frac{\phi(t^{-k}|\cdot |^{k})}{|\cdot|^{2}} \ast |\epsilon|^{2} \big ) \nabla \epsilon \big \|_{L_{x}^{2}}
\lesssim \| \epsilon \|_{\dot{H}_{x}^{1}}^{3} \ll t^{-k-1},
\end{equation}
for $t$ large. For the second term, H\"older's inequality and the Hardy-Littlewood -Sobolev inequality give us
\begin{equation}
\big \|  \big ( \frac{\phi(t^{-k}|\cdot |^{k})}{|\cdot|^{2}} \ast (\nabla |\epsilon|^{2}) \big ) \epsilon \big \|_{L_x^2} 
\lesssim \| \nabla (\epsilon\overline{\epsilon}) \| _{L_{x}^{4/3}} \| \epsilon \|_{L_{x}^{4}}
\lesssim \| \epsilon \|_{\dot{H}^{1}}^{2} \| \epsilon \|_{L_{x}^{2}} \ll t^{-k-1},
\end{equation}
for $t$ large. The remaining nonlinear terms of $F(t,x,\epsilon)$  in (\ref{eq:F}) can be estimated similarly. As for the $\pr_t \Qt$-term in $F(t,x,\epsilon)$, we note that $\| \pr_t \Qt \|_{H_x^1} \lesssim t^{-k-1}$ holds, by Lemma \ref{lem:tdecay} below. \end{proof}


\subsection{Completing the Proof of Proposition \ref{prop:epsilon}} 


We are now in position to conclude the proof of Proposition  \ref{prop:epsilon}. Assume $k \geq 8$ and choose $T_0 > 0$ sufficiently large. Proposition \ref{prop:iteration} and Lemma \ref{lem:F} imply the a-priori bounds on the iterates
\begin{equation}
||\epsilon_{q}(t,x)||_{H_{x}^{1}}\lesssim t^{-k+4} , \quad \mbox{for $t \geq T_0$},
\end{equation}
Passing to the difference equations (which eliminates the source terms $-i\partial_{t}Q^{(t)}$) and arguing identically to the above, one shows that $\{ \epsilon_{q} \}_{q = 0}^\infty$ forms a Cauchy sequence in $C^{0}_t H^1_x([T_{0}, \infty) \times \RR^4)$. Moreover, differentiating the equation and again recycling the same estimates, smoothness of the limit follows. Define the limit 
\begin{equation}
\epsilon := \lim_{q\rightarrow \infty}\epsilon_{q} .
\end{equation}
Next, we claim the bound 
\begin{equation}
|||x|\epsilon(t)||_{L_{x}^{2}}\lesssim t^{-k+5} , \quad \mbox{for $t \geq T_0$}.
\end{equation}
Indeed, it suffices to prove this bound for each iterate $\epsilon_{q}$. However, in view of \eqref{eps}, we have 
\begin{equation}
(i\partial_{t}+\DD-1)[x\epsilon_{q+1}(t,x)]=2\nabla_{x}\epsilon_{q+1}- \big( \frac{\phi(t^{-k}|\cdot|^{k})}{|\cdot|^{2}} \ast |\Qt|^{2} \big ) x\epsilon_{q+1}\pm\ldots,
\end{equation}
whence the desired bound follows from our a-priori bounds and integrating from backwards from $t=\infty$.

The proof of Proposition \ref{prop:epsilon} is now complete. As previously noted, this also completes the proof of Theorem \ref{th:main}. $\blacksquare$


\section{Finite codimensional stability of the conformal blow up} 
\label{sec:proof:BW}


This section is devoted to the proof of Theorem \ref{th:BW}. We thus consider $\eqref{eq:hartree}$ with\footnote{The factor $1/2\pi^2$ is a just convenient choice for this section and without loss of generality.} convolution kernel $\Phi(x)=\frac{1}{2 \pi^2 |x|^2}$, and we ask whether suitable perturbations of the initial data lead to the same blowup as for the explicit solution
\begin{equation} \label{eq:BW}
S(t,x)=t^{-2}e^{\frac{-ix^{2}}{4t}}e^{\frac{i}{t}}Q(\frac{x}{t}) ,
\end{equation}
where $Q=Q^{\infty} \in H^1(\RR^4)$ is the ground state satisfying 
\begin{equation} \label{eq:QBW}
\DD Q-Q+(-\DD)^{-1}(|Q|^{2})Q=0, \ \ Q=Q(|x|)>0 .
\end{equation}
Here and for the rest of this section, it is expedient to use the following notation: 
\begin{equation}
(-\DD)^{-1} f = \frac{1}{2 \pi^2} \big ( |x|^{-2} \ast f ) .
\end{equation}
Recall that Theorem \ref{th:Qunique} below (together with a simple scaling argument) ensures the uniqueness of the ground state $Q(|x|)$ solving (\ref{eq:QBW}).

This issue that blowup solutions of the form (\ref{eq:BW}) still persist under suitable perturbations of initial data was first addressed by Bourgain-Wang for $L^{2}$-critical NLS with local nonlinearities in \cite{Bourgain+Wang1997} for space dimensions $d=1,2$. There the authors show that one can construct blowup solutions, which decouple into the bulk part as above and a radiation part with suitable prescribed asymptotic profile at blowup time. More precisely, the profile has to belong to a finite-codimensional manifold. Here, we implement a similar procedure and consider the problem 
\begin{equation} \label{eq:BWt}
i\partial_{t}u+\DD u+(-\DD)^{-1}(|u|^{2})u=0 \quad \mbox{in $\RR^{4}$},
\end{equation}
which is equivalent to $\eqref{eq:hartree}$ with $d=4$ and $\Phi(x) = |x|^{-2}$, up to an inessential constant in front of the nonlinearity. For simplicity's sake, we first consider {\it{radial solutions}} of (\ref{eq:BWt}), and we later sketch the modifications needed for a more general result. 

In the spirit of \cite{Bourgain+Wang1997}, we try to find a solution of the form 
 \begin{equation}\label{BWAnsatz}
 u(t,x)=t^{-2}e^{\frac{-ix^{2}}{4t}}e^{\frac{i}{t}}[Q(\frac{x}{t})+\epsilon(t,x)]+z_{\psi}(t,x)  ,
 \end{equation}
 where the main perturbation $z_{\psi}(t,x)$ solves the initial-value problem 
 \begin{equation} \label{eq:ivp_z}
 \left \{ \begin{array}{l} i\partial_{t}z_{\psi}+\DD z_{\psi}+(-\DD)^{-1}(|z_{\psi}|^{2})z_{\psi}=0, \\[1ex]
 z_{\psi}(0,x)=\psi(x) . \end{array} \right .
 \end{equation}
Here the initial datum $\psi(x)\in C_{0}^{\infty}(\RR^4)$, say, satisfies a finite number of suitable vanishing conditions. Note that $z_{\psi}$ can always be constructed on some time interval $[-\delta_{0}, \delta_{0}]$ for $\delta_0 > 0$ sufficiently small. However, we immediately face a serious issue here: While the interactions of the bulk term in (\ref{BWAnsatz}) and the modified profile term $z_{\psi}$ can be made small by forcing sufficient vanishing of $\psi$ at the origin in the case of local NLS (see \cite{Bourgain+Wang1997}), this is {\em never true} for the Hartree equation (\ref{eq:BWt}). To see this, it suffices to consider terms of the form 
 \[
 \DD^{-1}(|z_{\psi}|^{2})t^{-2}e^{\frac{-ix^{2}}{4t}}e^{\frac{i}{t}}Q(\frac{x}{t}) \quad \mbox{and} \quad t^{-2}\DD^{-1}(Q^{2})(\frac{x}{t})z_{\psi}(t,x) .
 \]
The problem here is, of course, that the operator $\DD^{-1}$ destroys any localization properties of $z_{\psi}$. To deal with this, we use some modulation theory combined with the radiality assumption. Indeed, the strength of the interaction between the bulk term and the profile modifier $z_{\psi}$ due to the non-local character of the nonlinearity is seen to lead to non-trivial phase and scale shifts of the bulk term $t^{-2}e^{\frac{-ix^{2}}{4t}}e^{\frac{i}{t}}Q(\frac{x}{t})$, as can be seen from the statement of the following Proposition \ref{prop:BW}. Such shifts do not occur in the Bourgain-Wang method for the local NLS. 

In what follows, we introduce the notation 
\[
Q_{\lambda}(y)=\lambda^{2}Q(\lambda y), \quad \mbox{for $\lambda > 0$},
\]
so that in particular $Q=Q_{1}$ holds. Then Theorem \ref{th:BW} will be a direct consequence of the following result.
 
\begin{prop}\label{prop:BW} Let $\psi$ be a smooth radial profile flat near origin: $|\psi(x)|\lesssim |x|^{2N}$ for $N$ large enough. Assume that $\psi$ is small in the sense that $\psi(x)=\alpha\psi_{0}(x)$, where $\psi_{0}(x)\in C_{0}^{\infty}(\R^{4})$ is fixed and $|\alpha|\leq \alpha_{0}$, the latter small enough. Then, for $\delta_0 > 0$ sufficiently small, there exists $u \in C_t^0 H^1_x ([-\delta_0, 0) \times \RR^4)$ solving \eqref{eq:BWt} such that
\[
u(t,x)=e^{i(\frac{1}{t}+\gamma(\frac{1}{t}))}[\frac{1}{t^{2}}e^{\frac{-ix^{2}}{4t}}Q_{\lambda(\frac{1}{t})}(\frac{x}{t})+\epsilon(t,x)]+z_{\psi}(t,x),
\]
where $\epsilon,\gamma, \lambda$ satisfy  
\[ \mbox{$\| \epsilon(t) \|_{H^1_x} \to 0$, $\gamma(\frac{1}{t}) \to 0$, and $\lambda(\frac{1}{t}) \to 1$ as $t \nearrow 0$.} \]
Moreover, $z_\psi \in C_t^0 H^1_x ([-\delta_0, +\delta_0] \times \RR^4)$ solves the initial-value problem (\ref{eq:ivp_z}).
\end{prop}

By applying a pseudo-conformal transformation, the proof of Proposition \ref{prop:BW} will follow from:
\begin{prop}  \label{prop:BW2}
Under the assumptions of Proposition \ref{prop:BW}, there exists a solution of \eqref{eq:BWt} of the form
\begin{equation}\label{BWansatz}
v(s,y)=e^{i(s+\gamma(s))}[Q_{\lambda(s)}(y)+\epsilon(s,y)]+s^{-2}e^{\frac{iy^{2}}{4s}}z_{\psi}(\frac{1}{s}, \frac{y}{s})
\end{equation}
for $s\geq s_{0}$, with $s_{0} > 0$ sufficiently large, and we have the bounds 
\[
|\gamma(s)|\lesssim s^{-1},\, \| \epsilon(s,.) \|_{H^{1}_y}\lesssim s^{-3}, \, \| |y| \epsilon(s,.) \|_{L^2_y} \lesssim s^{-2},  \,|\lambda(s)-1|\lesssim s^{-3}, \quad \mbox{for $s \geq s_0$}.
\]
\end{prop} 

The rest of this section is devoted to the proof of Proposition \ref{prop:BW2}.


\subsection{Setting up the Iteration Scheme}


Let us derive the equation satisfied by $\epsilon$ when applying the ansatz (\ref{BWansatz}). To this end, we use a similar notation as in Section \ref{sec:proof:main}, and we write
\begin{equation}
L{\epsilon}=L_{+}\epsilon_{1}+iL_{-}\epsilon_{2},\quad \epsilon=\epsilon_{1}+i\epsilon_{2},
\end{equation}
with 
\begin{equation}
L_{-}=-\DD+ 1 - (\DD)^{-1}(Q^{2}),\quad L_{+}= - \DD + 1 - (\DD)^{-1}(Q^{2})-2 Q (\DD)^{-1}(Q .).
\end{equation}
We then obtain the following equation for $\epsilon(s,y)$:
\begin{equation}\label{epsmess}\begin{split}
(i\partial_{s}-L)\epsilon = &-\dot{\gamma}(s)[Q_{\lambda(s)}(y)+\epsilon]+i\dot{\lambda}(s)\lambda(s)[2+y\cdot\nabla]Q(\lambda(s)y)\\&+(-\DD)^{-1}(Q_{\lambda(s)}^{2}(y))s^{-2}e^{\frac{iy^{2}}{4s}}z_{\psi}(\frac{1}{s}, \frac{y}{s})+s^{-2}(-\DD)^{-1}(|z_{\psi}|^{2})(\frac{1}{s},\frac{y}{s})Q_{\lambda(s)}(y)\\&+(-\DD)^{-1}(2\Re(e^{i(s+\gamma(s))}Q_{\lambda(s)}(y)\overline{s^{-2}e^{\frac{iy^{2}}{4s}}z_{\psi}}(\frac{1}{s},\frac{y}{s})))Q_{\lambda(s)}(y)\\&+(-\DD)^{-1}(2\Re(e^{i(s+\gamma(s))}Q_{\lambda(s)}(y)\overline{s^{-2}e^{\frac{iy^{2}}{4s}}z_{\psi}}(\frac{1}{s},\frac{y}{s})))s^{-2}e^{\frac{iy^{2}}{4s}}e^{-i(s+\gamma(s))}z_{\psi}(\frac{1}{s}, \frac{y}{s})\\
&+A(\epsilon)+B(\epsilon^{2})+C(\epsilon^{3})+\Delta_{\lambda(s)}L\epsilon,
\end{split}\end{equation}
where 
\begin{equation}\nonumber\begin{split}
A(\epsilon)=&s^{-2}(-\DD)^{-1}(|z_{\psi}|^{2})(\frac{1}{s},\frac{y}{s})\epsilon(s,y) \\ & +
(-\DD)^{-1}(2\Re(e^{i(s+\gamma(s))}\epsilon(s,y)\overline{s^{-2}e^{\frac{iy^{2}}{4s}}z_{\psi}}(\frac{1}{s},\frac{y}{s})))Q_{\lambda(s)}(y)
\\&+(-\DD)^{-1}(2\Re(e^{i(s+\gamma(s))}Q_{\lambda(s)}(y)\overline{s^{-2}e^{\frac{iy^{2}}{4s}}z_{\psi}}(\frac{1}{s},\frac{y}{s})))\epsilon(s,y)
\\&+(-\DD)^{-1}(2\Re(e^{i(s+\gamma(s))}\epsilon(s,y)\overline{s^{-2}e^{\frac{iy^{2}}{4s}}z_{\psi}}(\frac{1}{s},\frac{y}{s})))s^{-2}e^{\frac{iy^{2}}{4s}}e^{-i(s+\gamma(s))}z_{\psi}(\frac{1}{s}, \frac{y}{s}),
\end{split}\end{equation}
as well as 
\begin{equation}\nonumber\begin{split}
B(\epsilon^{2})=&(-\DD)^{-1}(|\epsilon|^{2}(s,y))s^{-2}e^{\frac{iy^{2}}{4s}}z_{\psi}(\frac{1}{s}, \frac{y}{s})+(-\DD)^{-1}(|\epsilon|^{2}(s,y))Q_{\lambda(s)}(y)\\&+(-\DD)^{-1}(2\Re(Q_{\lambda(s)}(y)\overline{\epsilon}(s,y)))\epsilon(s,y)\\
&+(-\DD)^{-1}(2\Re(s^{-2}e^{\frac{iy^{2}}{4s}}z_{\psi}(\frac{1}{s}, \frac{y}{s})
\overline{e^{-i(s+\gamma(s))}\epsilon}(s,y)))\epsilon(s,y),
\end{split}\end{equation}
and finally
\[
C(\epsilon^{3})=(-\DD)^{-1}(|\epsilon|^{2})\epsilon .
\]
In the equation for $\epsilon(s,y)$, the last term $\Delta_{\lambda(s)}L\epsilon$ accounts for the error incurred upon replacing $Q_{\lambda(s)}$ by $Q=Q_{1}$ in the linear part, whence it is given by 
\begin{equation} \begin{split}
\Delta_{\lambda(s)}L\epsilon= & (-\DD)^{-1}(Q^{2}-Q_{\lambda(s)}^{2})\epsilon+\DD^{-1}(2\Re([Q-Q_{\lambda(s)}]\overline{\epsilon}))Q_{\lambda(s)} \\& +(-\DD)^{-1}(2\Re(Q\overline{\epsilon}))[Q-Q_{\lambda(s)}] .
\end{split}
\end{equation}
Note that the unknown $\epsilon$ also needs to vanish at infinity:
\[
\lim_{s\rightarrow \infty}||\epsilon(s,.)||_{H^{1}_x}=0.
\]
We shall find $\epsilon$, $\lambda$, $\gamma$ as the limits of a suitable iteration scheme. The purpose of the functions $\gamma(s)$, $\lambda(s)$ will be to partly eliminate the root part of the right hand side of \eqref{epsmess}. We first write \eqref{epsmess} in vectorial form as follows:
\begin{equation}\label{epsvector}\begin{split}
(i\partial_{s}+\calH)\bm\epsilon\\ \overline{\epsilon}\endm=&-\dot{\gamma}(s)\bm (Q_{\lambda(s)}(y)+\epsilon)\\- (Q_{\lambda(s)}(y)+\overline{\epsilon})\endm+\bm i\dot{\lambda}(s)\lambda(s)[2+y\cdot\nabla]Q(\lambda(s)y)\\ i\dot{\lambda}(s)\lambda(s)[2+y\cdot\nabla]Q(\lambda(s)y)\endm\\&+\bm F(Q_{\lambda(s)}, z_{\psi})\\ -\overline{F(Q_{\lambda(s)}, z_{\psi})}\endm+\bm A(\epsilon)+B(\epsilon^{2})+C(\epsilon^{3})+\DD_{\lambda(s)}L\epsilon\\ -\overline{[A(\epsilon)+B(\epsilon^{2})+C(\epsilon^{3})+\DD_{\lambda(s)}L\epsilon]}\endm .
\end{split}\end{equation}
Here the operator $\calH=\calH^{(\infty)}$ is the same as the one used in the preceding section (up to a simple rescaling due to the different choice of coupling constant in the nonlinearity). Also, the expression $F(Q_{\lambda(s)}, z_{\psi})$ refers to the sum of the terms three to six on the right hand side of \eqref{epsmess}.

Now, we assume that the iterates $\epsilon_{j}(s,y), \gamma_{j}(s), \lambda_{j}(s)$ have been defined, with bounds to be specified later. We need to specify how to choose  $\epsilon_{j+1}(s,y), \gamma_{j+1}(s), \lambda_{j+1}(s)$. Assuming $\gamma_{j+1}$, $\lambda_{j+1}$ to be chosen, we set 
\begin{equation}\label{epsvectorj}\begin{split}
&(i\partial_{s}+\calH)\bm\epsilon_{j+1}\\ \overline{\epsilon_{j+1}}\endm\\&=-\dot{\gamma}_{j+1}(s)\bm (Q_{\lambda_{j}(s)}(y)+\epsilon_{j})\\- (Q_{\lambda_{j}(s)}(y)+\epsilon_{j})\endm+\bm i\dot{\lambda}_{j+1}(s)\lambda_{j}(s)[2+y\cdot\nabla]Q(\lambda_{j}(s)y)\\ i\dot{\lambda}_{j+1}(s)\lambda_{j}(s)[2+y\cdot\nabla]Q(\lambda_{j}(s)y)\endm\\&+\bm F_{j}(Q_{\lambda_{j}(s)}, z_{\psi})\\ -\overline{F_{j}(Q_{\lambda_{j}(s)}, z_{\psi})}\endm+\bm A_{j}(\epsilon_{j})+B_{j}(\epsilon_{j}^{2})+C_{j}(\epsilon_{j}^{3})+\DD_{\lambda_{j}(s)}L\epsilon_{j}\\ -\overline{[A_{j}(\epsilon_{j})+B_{j}(\epsilon_{j}^{2})+C_{j}(\epsilon_{j}^{3})+\DD_{\lambda_{j}(s)}L\epsilon_{j}]}\endm .
\end{split}\end{equation}
Here the expression $F_{j}(Q_{\lambda_{j}(s)}, z_{\psi})$ is defined through the right-hand side of \eqref{epsmess}, but with $\lambda, \gamma$ replaced by $\lambda_{j}, \gamma_{j}$, and similarly $A_{j}(\epsilon_{j})$ is defined as $A(\epsilon)$ with $\epsilon, \gamma, \lambda$ replaced by $\epsilon_{j}, \gamma_{j}, \lambda_{j}$. 

Now we need to specify how to choose $\gamma_{j+1}, \lambda_{j+1}$. As for the former, we split 
\[
\gamma_{j+1}=\gamma_{1,j+1}+\gamma_{2,j+1} .
\]
Here the term $\gamma_{1,j+1}$ is chosen to essentially eliminate those terms on the right hand side of \eqref{epsvectorj} contained in 
\[
\bm F_{j}(Q_{\lambda_{j}(s)}, z_{\psi})\\ -\overline{F_{j}(Q_{\lambda_{j}(s)}, z_{\psi})}\endm,\,\bm A_{j}(\epsilon_{j})+B_{j}(\epsilon_{j}^{2})+C_{j}(\epsilon_{j}^{3})+\DD_{\lambda_{j}(s)}L\epsilon_{j}\\ -\overline{[A_{j}(\epsilon_{j})+B_{j}(\epsilon_{j}^{2})+C_{j}(\epsilon_{j}^{3})+\DD_{\lambda_{j}(s)}L\epsilon_{j}]}\endm ,
\]
which have $Q_{\lambda_{j}(s)}(y)$ as a third factor. Specifically, we define 
\begin{equation}\label{gamma1j+1}\begin{split}
\dot{\gamma}_{1,j+1}(s)= & s^{-2}(-\DD)^{-1}(|z_{\psi}|^{2})(\frac{1}{s},0) \\ &+(-\DD)^{-1}(2\Re(e^{i(s+\gamma_{j}(s))}\epsilon_{j}(s,y)\overline{s^{-2}e^{\frac{iy^{2}}{4s}}z_{\psi}}(\frac{1}{s},\frac{y}{s})))(s,0) .
\end{split}
\end{equation}
Note that with this choice of $\gamma_{1,j+1}$, the term 
\[
-\dot{\gamma}_{1,j+1}(s)\bm Q_{\lambda_{j}(s)}(y)\\- Q_{\lambda_{j}(s)}(y)\endm
\]
essentially cancels the terms on the right hand side of \eqref{epsvectorj} corresponding to the fourth term  in \eqref{epsmess} as well as the 2nd term in $A_{j}(\epsilon_{j})$. More precisely, we find that
\begin{equation}\nonumber\begin{split}
&[s^{-2}(-\DD)^{-1}(|z_{\psi}|^{2})(\frac{1}{s},0)+(-\DD)^{-1}(2\Re(e^{i(s+\gamma_{j}(s))}\epsilon_{j}(s,y)\overline{s^{-2}e^{\frac{iy^{2}}{4s}}z_{\psi}}(\frac{1}{s},\frac{y}{s})))(s,0)]Q_{\lambda_{j}(s)}(y)\\
&-[s^{-2}(-\DD)^{-1}(|z_{\psi}|^{2})(\frac{1}{s},\frac{y}{s})+(-\DD)^{-1}(2\Re(e^{i(s+\gamma_{j}(s))}\epsilon_{j}(s,y)\overline{s^{-2}e^{\frac{iy^{2}}{4s}}z_{\psi}}(\frac{1}{s},\frac{y}{s})))(s,y)]Q_{\lambda_{j}(s)}(y)\\&=O(\frac{1}{s^{N}}),
\end{split}\end{equation}
To see this, we note that we have uniform exponential decay $|Q_{\lambda}(x)| \lesssim e^{-c |x|}$ for some constant $c >0$, provided that $\lambda > 0$ varies in a compact set. The above estimate then follows from the radiality assumption and Newton's theorem, see e.\,g.~equation \eqref{Newton}, as well as the following elementary estimate which follows from finite Taylor expansion of $z_{\psi}$ with respect to $t$ and using the equation for $z_\psi$; see \cite{Bourgain+Wang1997} for a similar statement.

\begin{lemma} \label{lem:BW} Provided the initial condition $\psi(x) \in C^\infty_0(\RR^4)$ in \eqref{eq:ivp_z} satisfies $|\psi(x)|\lesssim |x|^{2N}$ for some $N$, we have 
\[
|z_{\psi}(t,x)|\lesssim \sum_{2l+j=2N}|t|^{l}|x|^{j},\quad \mbox{for $t \in [-\delta_{0}, \delta_{0}]$},
\]
with $\delta_0 >0$ a sufficiently small constant.
\end{lemma}
 In order to determine $\gamma_{2,j+1}, \lambda_{j+1}, \epsilon_{j+1}$, we now use the following iteration lemma, which also states the bounds:

\begin{lemma}\label{bwiteration} Let $\delta>0$ be small enough, and also $\alpha_{0}=\alpha_{0}(\delta)$ as in Proposition~\ref{prop:BW} small enough. Further assume $N$ large enough.Then, assuming the functions $\epsilon_{j}(s,y), \gamma_{j}(s),\lambda_{j}(s)$ to be $C^1$ and satisfying the bounds
\[
||\epsilon_{j}(s,.)||_{H^{1}}\leq \delta s^{-3},\,|\dot{\gamma}_{2,j}(s)|\leq \delta s^{-3},\, |\dot{\lambda}_{j}(s)|\leq \delta s^{-4},
\]
there exist $C^1$-functions $\gamma_{2,j+1}(s)$, $\lambda_{j+1}(s)$, $\epsilon_{j+1}(s,y)$, such that  if we define $\gamma_{j+1}=\gamma_{1,j+1}+\gamma_{2,j+1}$, then $\gamma_{j+1},\lambda_{j+1}, \epsilon_{j+1}$ satisfy \eqref{epsvectorj}. Furthermore, the functions  $\gamma_{2,j+1}(s)$, $\lambda_{j+1}(s)$, $\epsilon_{j+1}(s,y)$ satisfy identical bounds. 
\end{lemma}


\begin{proof}[Proof of Lemma \ref{bwiteration}]


Given a vector valued function $\bm F(s,y)\\ \overline{F(s,y)}\endm$, we shall invoke the decomposition
\[ 
\bm F(s,y)\\ \overline{F(s,y)}\endm=\bm F(s,y)\\ \overline{F(s,y)}\endm_{\text{root}}+\bm F(s,y)\\ \overline{F(s,y)}\endm_{\text{non-root}} .
\]
Here the root part is defined as in the equations following \eqref{eq:modul}, i.\,e.~we have 
\[
\bm F(s,y)\\ \overline{F(s,y)}\endm_{\text{root}}=\sum_{j=1}^{4}b_{j}(s)\phi_{j}(y),
\]
where the coefficients $b_{j}(s)$ are given by \eqref{root1}-\eqref{root4}.

Next,  we return to \eqref{epsvectorj} and rearrange the terms on the right-hand side as follows (recall that $Q=Q^{(\infty)}$):
\begin{equation}\label{epsvectorj'}\begin{split}
&(i\partial_{s}+\calH)\bm\epsilon_{j+1}\\ \overline{\epsilon_{j+1}}\endm=\\&
-\dot{\gamma}_{2,j+1}(s)\bm Q(y)\\ -Q(y)\endm+\bm i\dot{\lambda}_{j+1}(s)\lambda_{j}(s)[2+y\cdot\nabla]Q(y)\\ i\dot{\lambda}_{j+1}(s)\lambda_{j}(s)[2+y\cdot\nabla]Q(y)\endm\\&-\dot{\gamma}_{2,j+1}(s)\bm Q_{\lambda_{j}(s)}(y)-Q(y)+\epsilon_{j}\\ -[Q_{\lambda_{j}(s)}(y)-Q(y)+\overline{\epsilon_{j}}]\endm\\&
+\bm i\dot{\lambda}_{j+1}(s)\lambda_{j}(s)[2+y\cdot\nabla][Q(\lambda_{j}(s)y)-Q(y)]\\ i\dot{\lambda}_{j+1}(s)\lambda_{j}(s)[2+y\cdot\nabla][Q(\lambda_{j}(s)y)-Q(y)]\endm\\&+\bm F_{j}(Q_{\lambda_{j}(s)}, z_{\psi})\\ -\overline{F_{j}(Q_{\lambda_{j}(s)}, z_{\psi})}\endm-\dot{\gamma}_{1,j+1}(s)\bm Q_{\lambda_{j}(s)}(y)\\-\overline{Q_{\lambda_{j}(s)}(y)}\endm\\&+\bm A_{j}(\epsilon_{j})+B_{j}(\epsilon_{j}^{2})+C_{j}(\epsilon_{j}^{3})+\DD_{\lambda_{j}(s)}L\epsilon_{j}\\ -\overline{[A_{j}(\epsilon_{j})+B_{j}(\epsilon_{j}^{2})+C_{j}(\epsilon_{j}^{3})+\DD_{\lambda_{j}(s)}L\epsilon_{j}]}\endm-\dot{\gamma}_{1,j+1}(s)\bm \epsilon_{j}(s,y)\\ -\overline{\epsilon_{j}(s,y)}\endm
\end{split}\end{equation}
We can write this equation schematically as 
\begin{equation}\begin{split}
(i\partial_{s}+\calH)\bm\epsilon_{j+1}\\ \overline{\epsilon_{j+1}}\endm
=&-\dot{\gamma}_{2,j+1}(s)\bm Q(y)\\ -Q(y)\endm+\bm i\dot{\lambda}_{j+1}(s)\lambda_{j}(s)[2+y\cdot\nabla]Q(y)\\ i\dot{\lambda}_{j+1}(s)\lambda_{j}(s)[2+y\cdot\nabla]Q(y)\endm\\
&+\bm N(\lambda_{j},\lambda_{j+1}, \gamma_{j}, \gamma_{j+1}, \epsilon_{j})\\ -\overline{ N(\lambda_{j},\lambda_{j+1}, \gamma_{j}, \gamma_{j+1}, \epsilon_{j})}\endm
\end{split}\end{equation}
Now we apply the above decomposition into a root and non-root part to the last term on the right. Thus we write 
\begin{equation}\nonumber\begin{split}
&\bm N(\lambda_{j},\lambda_{j+1}, \gamma_{j}, \gamma_{j+1}, \epsilon_{j})\\ -\overline{ N(\lambda_{j},\lambda_{j+1}, \gamma_{j}, \gamma_{j+1}, \epsilon_{j})}\endm
=\bm N(\lambda_{j},\lambda_{j+1}, \gamma_{j}, \gamma_{j+1}, \epsilon_{j})\\ -\overline{ N(\lambda_{j},\lambda_{j+1}, \gamma_{j}, \gamma_{j+1}, \epsilon_{j})}\endm_{\text{non-root}}\\
&+i\alpha_{j}(s)\bm i Q\\ -iQ\endm+i\beta_{j}(s)\bm 2Q+y\cdot\nabla Q\\ 2Q+y\cdot\nabla Q\endm +i\tilde{\gamma}_{j}(s)\bm i |y|^{2}Q\\ -iQ|y|^{2}\endm+i\delta_{j}(s)\bm\rho\\ \rho\endm
\end{split}\end{equation}
Here the coefficients $\alpha_{j}(s)$ etc.~depend on $\lambda_{j},\lambda_{j+1}, \gamma_{j}, \gamma_{j+1}, \epsilon_{j}$, and are given by \eqref{root1}-\eqref{root4} applied to $N(\lambda_{j},\lambda_{j+1}, \gamma_{j}, \gamma_{j+1}, \epsilon_{j})$. Now we claim the following {\bf{bound}} 
\[
|\alpha_{j}(s)|+|\beta_{j}(s)|+|\tilde{\gamma}_{j}(s)|+|\delta_{j}(s)|\lesssim (\alpha_{0}+\delta^{2})s^{-6}+O(\delta s^{-3}(|\dot{\gamma}_{2,j+1}|+|\dot{\lambda}_{j+1}(s)|)) .
\]
To see this, we recall \eqref{root1}-\eqref{root4} and the that root modes satisfy a uniform exponential decay. Therefore it suffices to show, for some fixed $c>0$, 
\[
|\la N(\lambda_{j},\lambda_{j+1}, \gamma_{j}, \gamma_{j+1}, \epsilon_{j}), e^{-c|y|}\ra |
\lesssim (\alpha_{0}+\delta^{2})s^{-6}+O(\delta s^{-3}(|\dot{\gamma}_{2,j+1}|+|\dot{\lambda}_{j+1}(s)|)).
\]
To see this, we check this separately for the last four terms on the right-hand side of \eqref{epsvectorj'}:

\bigskip
\noindent
{\it{(1)}} We have 
\begin{equation}\nonumber\begin{split}
|\la\dot{\gamma}_{2,j+1}[ Q_{\lambda_{j}(s)}(y)-Q(y)+\epsilon_{j}], e^{-c|y|}\ra| &\lesssim|\dot{\gamma}_{2,j+1}(s)| (|\lambda_{j}(s)-1|+||\epsilon_{j}(s,.)||_{H^{1}})\\
&\lesssim |\dot{\gamma}_{2,j+1}(s)|\delta s^{-3} .
\end{split}\end{equation}
\\
{\it{(2)}} Similarly, we have 
\begin{equation}\nonumber\begin{split}
|\la i\dot{\lambda}_{j+1}(s)\lambda_{j}(s)[2+y\cdot\nabla][Q(\lambda_{j}(s)y)-Q(y)], e^{-c|y|}\ra|
\lesssim |\dot{\lambda}_{j+1}(s)| \delta s^{-3} .
\end{split}\end{equation}
\\
{\it{(3)}} Now consider the expression 
\[
\la F_{j}(Q_{\lambda_{j}(s)}, z_{\psi})+A_{j}(\epsilon_{j})-\dot{\gamma}_{1,j+1}(s) [Q_{\lambda_{j}(s)}(y)+\epsilon_{j}(s,y)], e^{-c|y|}\ra .
\]

\bigskip
As for the terms given by $F_{j}(Q_{\lambda_{j}(s)}, z_{\psi})$, i.\,e.~terms number three to six on the right-hand side of \eqref{epsmess}, as well as the terms constituting $A_{j}(\epsilon_{j})$, we deduce 
\[
|\la (-\DD)^{-1}(Q_{\lambda(s)}^{2}(y))s^{-2}e^{\frac{iy^{2}}{4s}}z_{\psi}(\frac{1}{s}, \frac{y}{s}), e^{-c|y|}\ra|\lesssim \alpha_{0}s^{-2-N} .
\]
This is of the desired form provided that $N\geq 4$. Next, the fourth term in \eqref{epsmess} is seen to combine with the second term in $A_{j}(\epsilon_{j})$ to essentially cancel against  $\dot{\gamma}_{1,j+1}(s)Q_{\lambda_{j}(s)}(y)$. That is, we find
\begin{equation}\nonumber\begin{split}
&|\la  s^{-2}(-\DD)^{-1}(|z_{\psi}|^{2})(\frac{1}{s},\frac{y}{s})Q_{\lambda(s)}(y)\\&+(-\DD)^{-1}(2\Re(e^{i(s+\gamma(s))}\epsilon_{j}(s,y)\overline{s^{-2}e^{\frac{iy^{2}}{4s}}z_{\psi}}(\frac{1}{s},\frac{y}{s})))Q_{\lambda(s)}(y)-\dot{\gamma}_{1,j+1}(s)Q_{\lambda_{j}(s)}(y), e^{-c|y|}\ra|\\&\lesssim \alpha_{0}s^{-N},
\end{split}\end{equation}
which is again as desired as long as $N\geq 6$. Note that we obtain the same type of cancellation for the expression 
\begin{equation}\nonumber\begin{split}
&\la  s^{-2}(-\DD)^{-1}(|z_{\psi}|^{2})(\frac{1}{s},\frac{y}{s})\epsilon_{j}(s,y)\\&+(-\DD)^{-1}(2\Re(e^{i(s+\gamma(s))}\epsilon_{j}(s,y)\overline{s^{-2}e^{\frac{iy^{2}}{4s}}z_{\psi}}(\frac{1}{s},\frac{y}{s})))\epsilon_{j}(s,y)-\dot{\gamma}_{1,j+1}(s)\epsilon_{j}(s,y), e^{-c|y|}\ra,
\end{split}\end{equation}
which contains the first and last term in $B_{j}(\epsilon_{j}^{2})$. Further, all terms in $F_{j}(Q_{\lambda_{j}(s)}, z_{\psi})$ and $A_{j}(\epsilon_{j})$ which contain a product 
\[
Q_{\lambda_{j}(s)}(y)z_{\psi}(\frac{1}{s}, \frac{y}{s})
\]
is again negligible, since it is necessarily of size $O(\alpha_{0}s^{-N})$, again acceptable if $N\geq 6$. 
\\

\noindent
{\it{(4)}} Using the bound for $\epsilon_{j}(s,y)$, we easily get 
\[
|\la B_{j}(\epsilon_{j}^{2})+C_{j}(\epsilon_{j}^{3})+\DD_{\lambda_{j}(s)}L\epsilon_{j}, e^{-c|y|}\ra|\lesssim \delta^{2} s^{-6} .
\]
This concludes the proof of the {\bf{bound}} above. 
\\

Continuing with the proof of the lemma, we now write 
\[
\bm\epsilon_{j+1}\\ \overline{\epsilon_{j+1}}\endm
=\bm\epsilon_{j+1}\\ \overline{\epsilon_{j+1}}\endm_{\text{root}}+\bm\epsilon_{j+1}\\ \overline{\epsilon_{j+1}}\endm_{\text{non-root}}
\]
Write 
\begin{equation}\nonumber\begin{split}
&\bm\epsilon_{j+1}\\ \overline{\epsilon_{j+1}}\endm_{\text{root}}\\&=
iA_{j+1}(s)\bm i Q\\ -iQ\endm+iB_{j+1}(s)\bm 2Q+y\cdot\nabla Q\\ 2Q+y\cdot\nabla Q\endm +i\Gamma_{j+1}(s)\bm i |y|^{2}Q\\ -iQ|y|^{2}\endm\\&+i\Delta_{j+1}(s)\bm\rho\\ \rho\endm
\end{split}\end{equation}
We will choose $\gamma_{2,j+1}(s)$, $\lambda_{j+1}(s)$ in such fashion that $A_{j+1}(s)$, $B_{j+1}(s)$ vanish. To solve for $\Delta_{j+1}(s)$, $\Gamma_{j+1}(s)$, we proceed as in the preceding section: 
we put 
\[
\partial_{s}\Delta_{j+1}(s)=\delta_{j}(s),\,\Delta_{j+1}(\infty)=0,
\]
as well as 
\[
\tilde{\gamma}_{j}(s)-\Delta_{j+1}(s)=\partial_{s}\Gamma_{j+1}(s),\,\Gamma_{j+1}(\infty)=0 .
\]
Note that
\[
|\Gamma_{j+1}(s)|+|\Delta_{j+1}(s)|\lesssim (\alpha_{0}+\delta^{2})s^{-4}+O(\delta s^{-1}(|\dot{\gamma}_{2,j+1}|+|\dot{\lambda}_{j+1}|) .
\]
Furthermore, our choices for $\Gamma_{j+1}(s)$, $\Delta_{j+1}(s)$ imply that 
\begin{equation}\nonumber\begin{split}
&(i\partial_{s}+\calH)[\Delta_{j+1}(s)\bm\rho\\ \rho\endm+\Gamma_{j+1}(s)\bm i |y|^{2}Q\\ -i|y|^2 Q\endm]
\\&=i\tilde{\gamma}_{j}(s)\bm i |y|^{2}Q\\ -i|y|^2 Q\endm+i\delta_{j}(s)\bm\rho\\ \rho\endm
-4i\Gamma_{j+1}(s)\bm 2Q+y\cdot\nabla Q\\  2Q+y\cdot\nabla Q\endm
\end{split}\end{equation}
Now we choose $\lambda_{j+1}$, $\gamma_{2,j+1}$ as follows:
\[
\dot{\lambda}_{j+1}(s)+\beta_{j}(s)-4\Gamma_{j+1}(s)=0,\,\lambda_{j+1}(\infty)=1,
\]
\[
-\dot{\gamma}_{2,j+1}(s)+\alpha_{j}(s)=0, \gamma_{2,j+1}(\infty)=0 .
\]
Of course, the functions $\Gamma_{j+1}(s)$, $\beta_{j}(s)$, $\alpha_{j}(s)$ depend implicitly and linearly (but with small coefficient) on $\dot{\lambda}_{j+1}(s), \dot{\gamma}_{2,j+1}(s)$, so we use the implicit function theorem here to solve these equations. We immediately obtain the bound 
\[
|\dot{\lambda}_{j+1}(s)|\lesssim (\alpha_{0}+\delta^{2})s^{-4},\,|\dot{\gamma}_{2,j+1}(s)|\lesssim (\alpha_{0}+\delta^{2})s^{-6} .
\]
We then set $A_{j+1}(s)=0$, $B_{j+1}(s)=0$, and upon setting 
\[
\bm\epsilon_{j+1}\\ \overline{\epsilon_{j+1}}\endm=\Delta_{j+1}(s)\bm\rho\\ \rho\endm+\Gamma_{j+1}(s)\bm i |y|^{2}Q\\ -i|y|^2 Q\endm+\bm\epsilon_{j+1}\\ \overline{\epsilon_{j+1}}\endm_{\text{non-root}},
\]
reduce to solving 
\[
(i\partial_{s}+\calH)\bm\epsilon_{j+1}\\ \overline{\epsilon_{j+1}}\endm_{\text{non-root}}=
\bm N(\lambda_{j},\lambda_{j+1}, \gamma_{j}, \gamma_{j+1}, \epsilon_{j})\\ -\overline{ N(\lambda_{j},\lambda_{j+1}, \gamma_{j}, \gamma_{j+1}, \epsilon_{j})}\endm_{\text{non-root}} .
\]
This we do as in the preceding section by setting
\[
\bm\epsilon_{j+1}\\ \overline{\epsilon_{j+1}}\endm_{\text{non-root}}(t,.)=\int_{t}^{\infty}e^{i(t-s)\calH}\bm N(\lambda_{j},\lambda_{j+1}, \gamma_{j}, \gamma_{j+1}, \epsilon_{j})\\ -\overline{ N(\lambda_{j},\lambda_{j+1}, \gamma_{j}, \gamma_{j+1}, \epsilon_{j})}\endm_{\text{non-root}}(s,.)ds .
\]
Recall that the operator $e^{it\calH}$ acts bounded in the $H^{1}$ sense on functions which project trivially onto the root part, see e.\,g.~the proof of Proposition~\ref{prop:subiteration}. We now establish the bound 
\[
||\bm N(\lambda_{j},\lambda_{j+1}, \gamma_{j}, \gamma_{j+1}, \epsilon_{j})\\ -\overline{ N(\lambda_{j},\lambda_{j+1}, \gamma_{j}, \gamma_{j+1}, \epsilon_{j})}\endm_{\text{non-root}}(s,.)||_{H^{1}}\lesssim (\alpha_{0}+\delta^{2})s^{-4} .
\]
This we do by treating the various components of $\bm N(\lambda_{j},\lambda_{j+1}, \gamma_{j}, \gamma_{j+1}, \epsilon_{j})\\ -\overline{ N(\lambda_{j},\lambda_{j+1}, \gamma_{j}, \gamma_{j+1}, \epsilon_{j})}\endm$:
\\
\bigskip
\noindent
{\it{(1)}} The first term, which corresponds to the fourth expression in \eqref{epsvectorj'}, is estimated by 
\begin{equation}\nonumber\begin{split}
&||\dot{\gamma}_{2,j+1}(s)\bm Q_{\lambda_{j}(s)}(y)-Q(y)+\epsilon_{j}\\ -[Q_{\lambda_{j}(s)}(y)-Q(y)+\overline{\epsilon_{j}}]\endm||_{H^{1}}\lesssim |\dot{\gamma}_{2,j+1}(s)| (|\lambda_{j}(s)-1|+||\epsilon_{j}(s,.)||_{H^{1}})\\&\lesssim 
(\alpha_{0}+\delta^{2})s^{-8} .
\end{split}\end{equation}
{\it{(2)}} Similarly, we estimate 
\begin{equation}\nonumber\begin{split}
&||\bm i\dot{\lambda}_{j+1}(s)\lambda_{j}(s)[2+y\cdot\nabla][Q(\lambda_{j}(s)y)-Q(y)]\\ i\dot{\lambda}_{j+1}(s)\lambda_{j}(s)[2+y\cdot\nabla][Q(\lambda_{j}(s)y)-Q(y)]\endm||_{H^{1}}
\lesssim |\dot{\lambda}_{j+1}(s)| (|\lambda_{j}(s)-1)\\
&\lesssim (\alpha_{0}+\delta^{2})s^{-7} .
\end{split}\end{equation}
\\
{\it{(3)}} Recalling the constituents of $F_{j}(Q_{\lambda_{j}(s)}(y), z_{\psi})$, we have 
\[
||(-\DD)^{-1}(Q_{\lambda(s)}^{2}(y))s^{-2}e^{\frac{iy^{2}}{4s}}z_{\psi}(\frac{1}{s}, \frac{y}{s})||_{H^{1}}\lesssim \alpha_{0}s^{-4} .
\]
Further, as in the proof of the {\bf{bound}} further above we take advantage of our choice of $\dot{\gamma}_{1,j+1}$ to estimate 
\begin{equation}\nonumber\begin{split}
&|| s^{-2}(-\DD)^{-1}(|z_{\psi}|^{2})(\frac{1}{s},\frac{y}{s})Q_{\lambda(s)}(y)\\&+(-\DD)^{-1}(2\Re(e^{i(s+\gamma(s))}\epsilon_{j}(s,y)\overline{s^{-2}e^{\frac{iy^{2}}{4s}}z_{\psi}}(\frac{1}{s},\frac{y}{s})))Q_{\lambda(s)}(y)-\dot{\gamma}_{1,j+1}(s)Q_{\lambda_{j}(s)}(y)||_{H^{1}}\\&\lesssim \alpha_{0}s^{-N},
\end{split}\end{equation}
which is acceptable if $N\geq 4$. The same argument applies to 
\begin{equation}\nonumber\begin{split}
&s^{-2}(-\DD)^{-1}(|z_{\psi}|^{2})(\frac{1}{s},\frac{y}{s})\epsilon_{j}(s,y)\\&+(-\DD)^{-1}(2\Re(e^{i(s+\gamma(s))}\epsilon_{j}(s,y)\overline{s^{-2}e^{\frac{iy^{2}}{4s}}z_{\psi}}(\frac{1}{s},\frac{y}{s})))\epsilon_{j}(s,y)-\dot{\gamma}_{1,j+1}(s)\epsilon_{j}(s,y)
\end{split}\end{equation}
Also, as in the proof of the {\bf{bound}} further above, all terms in $F_{j}(Q_{\lambda_{j}(s)}(y), z_{\psi})$ as well as $A_{j}(\epsilon_{j})$ which contain a product $Q_{\lambda_{j}(s)}(y)z_{\psi}(\frac{1}{s}, \frac{y}{s})
$ are of size $O(\alpha_{0}s^{-N})$  when estimated with respect to the $||.||_{H^{1}}$-norm. \\

\noindent
{\it{(4)}} The terms at least quadratic in $\epsilon_{j}$ are all of size $O(\delta^{2}s^{-6})$. We consider here the cubic term $C(\epsilon_{j}^{3})$. There we can estimate 
\[
||\DD^{-1}(|\epsilon_{j}^{2}|)\epsilon_{j}||_{\dot{H}^{1}}
\lesssim ||\DD^{-1}2\Re(\nabla\epsilon_{j}\overline{\epsilon_{j}})\epsilon_{j}||_{L^{2}}+
||\DD^{-1}(|\epsilon_{j}|^{2})\nabla\epsilon_{j}||_{L^{2}} .
\]
For the first term on the right, we have 
\[
 ||\DD^{-1}2\Re(\nabla\epsilon_{j}\overline{\epsilon_{j}})\epsilon_{j}||_{L^{2}}
 \lesssim ||\DD^{-1}2\Re(\nabla\epsilon_{j}\overline{\epsilon_{j}})||_{L^{4}}||\epsilon_{j}||_{L^{4}}
\lesssim ||\nabla\epsilon_{j}||_{L_{x}^{2}}|\epsilon_{j}||_{\dot{H}^{1}}^{2}\lesssim \delta^{3}s^{-9} ,
\]
where we have used the Sobolev embedding. Furthermore, by Hardy's inequality, we obtain
\[
||\DD^{-1}(|\epsilon_{j}|^{2})\nabla\epsilon_{j}||_{L^{2}}\lesssim ||\epsilon_{j}||_{\dot{H}^{1}}^{3}\lesssim \delta^{3}s^{-9} .
\]
The expression $||\DD^{-1}(|\epsilon_{j}^{2}|)\epsilon_{j}||_{L^{2}}$ is estimated similarly. \\

\medskip
\noindent
{\it{(5)}} The term $\DD_{\lambda_{j}(s)}L\epsilon_{j}$ is estimated by 
\[
||\DD_{\lambda_{j}(s)}L\epsilon_{j}||_{H^{1}}\lesssim ||\epsilon_{j}||_{H^{1}}|\lambda_{j}(s)-1|\lesssim \delta^{2}s^{-6} .
\]

The proof of Lemma \ref{bwiteration} is now concluded by choosing $\alpha_{0}, \delta$ small enough in the bounds for $\lambda_{j+1}, \gamma_{2,j+1}$, $\epsilon_{j+1}(s,y)$. \end{proof}


\subsection{Proof of Proposition~\ref{prop:BW2}}


In order to complete the proof of Proposition~\ref{prop:BW2}, we need the convergence of the $\epsilon_{j}, \gamma_{j}, \lambda_{j}$. This follows from the next lemma
\begin{lemma}Put $\epsilon_{0}(s,y)=0$, $\lambda_{0}(s)=1$, $\gamma_{0}(s)=0$, and define $\epsilon_{j}(s,y), \gamma_{j}(s), \lambda_{j}(s)$ inductively as above. Then if $\alpha_{0}>0$ is chosen small enough, the sequence $\{\epsilon_{j}(s,y)\}_{j\geq 0}$ converges in the $H^{1}$-topology, and satisfies uniform bounds 
\[
||\epsilon_{j}(s,y)||_{H^{1}}\leq \delta s^{-3}
\]
for suitable $\delta=\delta(\alpha_{0})$. Furthermore, the sequences $\{\lambda_{j}(s)\}_{j\geq 0}, \{\gamma_{j}(s)\}_{j\geq 0}$ converge in the uniform $C^{1}$ topology and satisfy the bounds 
\[
|\dot{\gamma}_{j}(s)|\leq \delta s^{-3},\,|\dot{\lambda}_{j+1}|\leq \delta s^{-3} .
\]
\end{lemma}

The only statement contained in the lemma that requires proof is the convergence of the iterates. However, this follows in standard fashion by forming the equations for the differences $\epsilon_{j+1}-\epsilon_{j}$ etc.~and repeating the arguments in the proof of the preceding lemma. Details are omitted. 
\\

Now let $\epsilon(s,y):=\lim_{j\rightarrow\infty}\epsilon_{j}(s,y)$, and similarly for $\lambda(s), \gamma(s)$. To conclude the proof of Proposition~\ref{prop:BW2}, we need to establish the bound 
\[
|| y \epsilon(s,.)||_{L^{2}}\lesssim s^{-2} .
\]
However, for this we note that $y \epsilon(s,.)$ satisfies the equation 
\begin{equation}\nonumber\begin{split}
&(i\partial_{s}+\calH)\bm y\epsilon\\ \overline{y\epsilon}\endm\\&=-\dot{\gamma}(s)\bm y(Q_{\lambda
(s)}(y)+\epsilon)\\- y(Q_{\lambda(s)}(y)+\epsilon)\endm+\bm yi\dot{\lambda}(s)\lambda(s)[2+y\cdot\nabla]Q(\lambda(s)y)\\ yi\dot{\lambda}(s)\lambda(s)[2+y\cdot\nabla]Q(\lambda(s)y)\endm\\&+\bm yF(Q_{\lambda(s)}, z_{\psi})\\ -\overline{yF(Q_{\lambda(s)}, z_{\psi})}\endm+\bm y[A(\epsilon)+B(\epsilon^{2})+C(\epsilon^{3})+\DD_{\lambda(s)}L\epsilon]\\ -\overline{y[A(\epsilon)+B(\epsilon^{2})+C(\epsilon^{3})+\DD_{\lambda(s)}L\epsilon]}\endm+2\bm \nabla\epsilon\\ -\overline{\nabla\epsilon}\endm
\end{split}\end{equation}
Then we replace the operator $\calH$ by the more elementary $\bm \DD-1& 0\\ 0&-\DD+1\endm$ and moving the local linear terms over to the right-hand side. Thus we get the equation
\begin{equation}\nonumber\begin{split}
&(i\partial_{s}+\bm \DD-1& 0\\ 0&-\DD+1\endm)\bm y\epsilon\\ \overline{y\epsilon}\endm\\&=-\dot{\gamma}(s)\bm y(Q_{\lambda
(s)}(y)+\epsilon)\\- y(Q_{\lambda(s)}(y)+\epsilon)\endm+\bm yi\dot{\lambda}(s)\lambda(s)[2+y\cdot\nabla]Q(\lambda(s)y)\\ yi\dot{\lambda}(s)\lambda(s)[2+y\cdot\nabla]Q(\lambda(s)y)\endm\\&+\bm yF(Q_{\lambda(s)}, z_{\psi})\\ -\overline{yF(Q_{\lambda(s)}, z_{\psi})}\endm+\bm y[A(\epsilon)+B(\epsilon^{2})+C(\epsilon^{3})+\DD_{\lambda(s)}L\epsilon]\\ -\overline{y[A(\epsilon)+B(\epsilon^{2})+C(\epsilon^{3})+\DD_{\lambda(s)}L\epsilon]}\endm+2\bm \nabla\epsilon\\ -\overline{\nabla\epsilon}\endm\\
&-\bm \DD^{-1}(Q^{2})y\epsilon\\ -\overline{ \DD^{-1}(Q^{2})y\epsilon}\endm-2\bm \DD^{-1}\Re(Qy\epsilon)Q\\ -\DD^{-1}\Re(Qy\epsilon)Q\endm
\end{split}\end{equation}

The bound for $||y\epsilon||_{L^{2}}$ is now obtained by making the bootstrapping assumption 
\[
||y\epsilon(s,y)||_{L_{y}^{2}}\leq \delta s^{-2}
\]
and recovering it by repeating the same estimates as in lemma~\ref{bwiteration}, as well as the already established bound for $||\nabla\epsilon||_{L^{2}}$, and using elementary bounds for the linear evolution of $\bm \DD-1& 0\\ 0&-\DD+1\endm$. This completes the proof of Proposition~\ref{prop:BW2}. $\blacksquare$

\subsection*{Remark about Nonradial Setting}

We mention here that essentially the same method presented above carries over to the non-radial case. Note that one has additional root modes due to the presence of translational symmetries as well as Galilei transformations, but these do not increase the algebraic degeneracy of the linear operator $\calH$. The main difference has to do with the formulation of the vanishing conditions on $\psi(x)$. Note that we crucially use the high-order vanishing of the expression
\[
\DD^{-1}(|z_{\psi}|^{2})(\frac{1}{s}, \frac{y}{s})-\DD^{-1}(|z_{\psi}|^{2})(\frac{1}{s}, 0) 
\]
at the origin $y=0$, which in turn follows in the radial case by simply forcing sufficient vanishing for $z_{\psi}$ at the origin. This is a consequence of Taylor expanding $z_{\psi}$. In the non-radial setting, we then arrive at the conditions 
\[
\nabla_{y}^{k}\partial_{s}^{l}\DD^{-1}(|z_{\psi}|^{2})(0,0)=0,\quad \mbox{for $1\leq k\leq N,\,0\leq l\leq N$}.
\]
Of course, these conditions can be formulated purely in terms of $\psi(x)$ upon using the equation to eliminate the operators $\partial_{s}$. Needless to say, these conditions appear to be rather cumbersome nonlinear vanishing conditions. For example, the simplest one corresponding to $k=1$, $l=0$ is given by $\DD^{-1}(\Re(\nabla \psi \overline{\psi}))(0,0)=0.$

\section{Ground States and Spectral Properties} \label{sec:spectral}

In this section, we consider ground state solutions $Q \in H^1(\RR^4)$ of
\begin{equation} \label{eq:Qground}
-\DD Q  + Q - \big (|x|^{-2} \ast |Q|^2 \big ) Q =  0.
\end{equation}
Apart from uniqueness of $Q$ (which follows from adapting an argument by E.~Lieb in \cite{Lieb1977}), we are mainly concerned with proving spectral properties of the linearized operator
\begin{equation}
L_+ = -\DD + 1 - \big ( |x|^{-2} \ast |Q|^2 \big ) - 2 Q \big ( |x|^{-2} \ast (Q \cdot) \big ) .
\end{equation} 
Due to the nonlocal term in $L_+$, standard ODE methods (like Sturm's oscillation theorem) are not our disposal to study the behavior of radial eigenfunctions of $L_+$. By contrast, such standard arguments play an important in the study of linearized operators for NLS with {\em local} nonlinearities. Here, however, we have to use the very structure of $L_+$, reflected by the fact $|x|^{-2}$ is (up to multiplicative constant) the Green's function of $\Delta$ in $\RR^4$. The technical main result of this section will be Theorem \ref{th:ker} below, which states that
\begin{equation}
\mathrm{ker} \, L_+ = \{ 0 \} \quad \mbox{when $L_+$ acts on $L^2_{\rm rad}(\RR^4)$ } .
\end{equation}
With the help of this nondegeneracy result, the technical Lemmas \ref{lem:H_inf} and \ref{lem:H_t} (used in Section \ref{sec:proof:main} above) about the matrix operators $\Hinf$ and $\Ht$ then follow by standard arguments, as detailed in Subsections \ref{subsec:Hinf} and \ref{subsec:Ht}, respectively. Also, the nondegeneracy of $L_+$ plays an essential role in Section \ref{sec:Qt} when constructing the modified profiles $\Qt$ by means of an implicit function type argument.

\subsection{Uniqueness of Ground States}

We have the following result.

\begin{theorem} \label{th:Qunique}
The equation 
\begin{equation} \label{eq:Qground}
-\DD Q + Q - \big ( \frac{1}{|x|^2} \ast |Q|^2 \big ) Q = 0
\end{equation}
has a unique positive, radial solution $Q(r) > 0$ in $H^1(\RR^4)$.
\end{theorem}

\begin{proof}
Existence of a radial, positive solution $Q \in H^1(\RR^4)$ follows from well-known arguments. For example, by standard variational methods and rearrangement inequalities, we deduce that there existst a radial, positive $Q \in H^1(\RR^4)$ such that
\begin{equation}
J(Q) = \inf_{f \in H^1(\RR^4) \setminus \{ 0 \} } J(f) , 
\end{equation}
where $J(f)$ is the Weinstein functional given by
\begin{equation}
J(f) = \frac{ \| \nabla f \|_{L^2_x}^2 \| f \|_{L^2_x}^2 }{ \int_{\RR^4} ( |x|^{-2} \ast |u|^2 ) |u|^2 \, dx } .
\end{equation}
One easily checks that any minimizer $Q(r) > 0$ satisfies (\ref{eq:Qground}) after a suitable rescaling $Q(r) \mapsto \alpha Q(\beta r)$ with some $\alpha, \beta > 0$.

The uniqueness proof, however, strongly depends on specific features of equation (\ref{eq:Qground}). Here, by adapting E.~Lieb's uniqueness proof in \cite{Lieb1977} for ground states $\phi \in H^1(\RR^3)$ of the Choquard-Pekar equation (in $d=3$ dimensions)
\begin{equation}
-\DD \phi + \phi - \big ( \frac{1}{|x|} \ast |\phi|^2 \big ) \phi = 0 \quad \mbox{in $\RR^3$},
\end{equation}
we can deduce the analogous result for (\ref{eq:Qground}) in $d=4$ dimensions. For the reader's convenience and also for later use, we now present our adaptation of Lieb's uniqueness proof to equation (\ref{eq:Qground}) with some modifications.

\medskip
Recall that Newton's theorem in $\RR^4$ says (note that $2 \pi^2$ is the area of the unit sphere in $\RR^4$):
\begin{equation}\label{Newton}
- \big ( |x|^{-2} \ast \rho ) (r) = \int_0^r K(r,s) \rho(s) \, ds - 2 \pi^2 \int_0^\infty \rho(s) s \, ds ,
\end{equation}
for radial functions $\rho = \rho(|x|)$ on $\RR^4$. Here 
\begin{equation} \label{eq:Kdef}
K(r,s) = 2 \pi^2 s \Big ( 1 - \frac{s^2}{r^2} \Big ) \geq 0, \quad \mbox{for $r \geq s$}.
\end{equation}
Hence, by Newton's theorem, we find that equation (\ref{eq:Qground}) for radial, real-valued $Q \in H^1(\RR^4)$ can be written as
\begin{equation}  \label{eq:Qground2}
-Q''(r) - \frac{3}{r} Q'(r) + \Big ( \int_0^r K(r,s) Q(s)^2 \, ds \Big ) Q(r) = e Q(r) ,
\end{equation}
where
\begin{equation} \label{eq:ee}
e = -1 +  2 \pi^2 \int_0^\infty Q(s)^2 s \, ds  > 0 .
\end{equation}
Note that $e >0$ follows from multiplying equation (\ref{eq:Qground}) by $Q \not \equiv 0$, integrating, and using that $K(r,s) \geq 0$ holds. Furthermore, by rescaling $Q(r) \mapsto e^{-1} Q(e^{-1/2} r)$ we can assume without loss of generality that $e=1$ holds.

Let us now suppose that $Q(r) > 0$ and $R(r) > 0$ are two positive, radial solutions of (\ref{eq:Qground}) in $H^1(\RR^4)$ such that $Q \not \equiv R$. As previously remarked, we can assume (after a rescaling) that both $Q$ and $R$ satisfy (\ref{eq:Qground2}) with $e =1$. Therefore $Q(r)$ and $R(r)$ solve the initial-value problem
\begin{equation} \label{eq:ivp}
\left \{ \begin{array}{l} \displaystyle -u''(r) - \frac{3}{r} u'(r) - u(r) + \big ( \int_0^r K(r,s) u(s)^2 \, ds \big ) u(r) = 0, \\[1ex]
u(0) = u_0, \quad u'(0) = 0 , \end{array} \right .
\end{equation}
with initial conditions $u(0) = Q(0)$ and $u(0) = R(0)$, respectively. A standard fixed-point argument shows that the initial-value problem (\ref{eq:ivp}) has unique local $C^2$-solution $u(r)$ for given $u_0 \in \RR$. Moreover, the corresponding solution $u(r)$ exists up to some maximal radius of existence $r_{\rm max} \in (0, \infty]$. In particular, we deduce that $Q(0) \neq R(0)$ must hold, since otherwise $Q \equiv R$.

Therefore we can henceforth assume that $Q(0) > R(0)$ holds, say. Then, by continuity, we have $Q(r) > R(r)$ at least on some initial interval. We now claim that in fact 
\begin{equation} \label{eq:QRlarge}
Q(r) > R(r)  \quad \mbox{for all $r \geq 0$}.
\end{equation}
To show this, we introduce the functions
\begin{equation}
U_Q(r) = \int_0^r K(r,s) Q(s)^2 \,ds \quad \mbox{and} \quad U_R(r) = \int_0^r K(r,s) R(s)^2 \, ds.
\end{equation}
Then an elementary calculation using the equation in (\ref{eq:ivp}) yields the ``Wronskian--type'' identity
\begin{equation}
\frac{d}{dr} \big \{ r^3 \big ( Q' R - Q R' ) \big \} = r^3 Q R \big (U_Q - U_R \big ) ,
\end{equation}
which, by integration, gives us
\begin{equation} \label{eq:wronsk}
r^3 \big ( Q'(r) R(r) - Q(r) R'(r) \big ) = \int_0^r s^3 Q(s) R(s) \big ( U_Q(s) - U_R(s) \big ) \, ds
\end{equation}
Next, we suppose that (\ref{eq:QRlarge}) fails to hold, i.\,e., the function $Q(r)$ intersects $R(r)$ for the first time at $r_* > 0$, say. Then the left-hand side of (\ref{eq:wronsk}) at $r=r_*$ satisfies
\begin{equation}
r_*^3 Q(r_*) ( Q'(r_*) - R'(r_*) ) \leq 0,
\end{equation}
whereas the right-hand side of (\ref{eq:wronsk}) must obey
\begin{equation}
\int_0^{r_*} s^3 Q(s) R(s) \big (U_Q(s) - U_R(s) ) \, ds > 0,
\end{equation}
since $Q(r) > 0$ and $R(r) > 0$, as well as $U_Q(r) > U_R(r)$ for $0 < r < r_*$. This contradiction shows that $Q(r)$ and $R(r)$ can never intersect and hence (\ref{eq:QRlarge}) must hold.

It remains to show that (\ref{eq:QRlarge}) also leads to a contradiction, which can be seen as follows. Consider the Schr\"odinger operators
\begin{equation}
H_Q = -\DD + U_Q \quad \mbox{and} \quad H_R = -\DD + U_R .
\end{equation}
Clearly, the strictly positive functions $Q$ and $R$ are (normalized) ground states (with eigenvalue $e=1$) for $H_Q$ and $H_R$, respectively. Therefore,
\begin{equation} \label{eq:HR}
\langle \phi, H_Q \phi \rangle \geq \| \phi \|_{L^2}^2 \quad \mbox{and} \quad \langle \phi, H_R \phi \rangle \geq \| \phi \|_{L^2}^2 
\end{equation}
for all $\phi \in H^1(\RR^4)$. Moreover, by standard arguments, we have uniqueness of ground states for $H_Q$ and $H_R$, so that equality in (\ref{eq:HR}) holds if and only if $\phi = \lambda Q$ or $\phi = \lambda R$ for some constant $\lambda$, respectively. Since equation (\ref{eq:QRlarge}) implies that $U_Q(r) > U_R(r)$ for all $r > 0$, we deduce from (\ref{eq:HR}) that
\begin{align*}
\| Q \|_{L^2}^2 & \leq \langle Q, H_R Q \rangle = \langle Q, H_Q Q \rangle - \langle Q, (U_Q-U_R) Q \rangle = \| Q \|_{L^2}^2 - \delta
\end{align*}
with some $\delta > 0$, which is a contradiction. 

This shows that equation (\ref{eq:Qground2}) cannot have two distinct positive, radial nontrivial solutions $Q \in H^1(\R^4)$. This completes the proof of Theorem \ref{th:Qunique}. \end{proof}

\subsection*{Nondegeneracy of $L_+$}

By linearizing equation (\ref{eq:Qground}) on the space of real-valued functions, we obtain the scalar nonlocal, self-adjoint operator
\begin{equation} \label{eq:Lplus}
L_+ = -\DD + 1 - \big ( |x|^{-2} \ast |Q|^2 \big ) - 2 Q \big ( |x|^{-2} \ast (Q \cdot) \big )
\end{equation}
acting on $L^2(\RR^4)$ with domain $H^2(\RR^4)$. We will now prove the important fact that $L_+$ has trivial kernel in the radial sector.

\begin{theorem} \label{th:ker}
Let $L_+$ be given by (\ref{eq:Lplus}), where $Q \in H^1(\RR^4)$ is the unique radial, positive solution from Theorem \ref{th:Qunique}. Then we have 
\[
\mathrm{ker} \, L_+ = \{ 0 \} \quad \mbox{when $L_+$ acts on $L^2_{\rm rad}(\RR^4)$.}
\]
\end{theorem}

\begin{proof} 
We argue by contradiction. Suppose that $\xi \in L^2_{\mathrm{rad}}(\RR^4)$ with $\xi \not \equiv 0$ satisfies
\begin{equation} \label{eq:xi}
L_+ \xi = 0 ,
\end{equation}
which, by simple bootstrap arguments, implies that $\xi \in H^k(\RR^4)$ for all $k \geq 0$. Furthermore, by Newton's theorem, we can write the left-hand side in (\ref{eq:xi}) as
\begin{equation} \label{eq:LL}
L_+ \xi = \mathcal{L}_+ \xi - 2 Q \big ( \int_{\RR^4} \frac{Q \xi}{|x|^2} \,dx \big ) ,
\end{equation}
where $\mathcal{L}_+$ is the linear operator given by
\begin{align}
(\mathcal{L}_+ v)(r) & = -v''(r) - \frac{3}{r} v'(r) + e v(r) + V(r) v(r)  \nonumber \\
& \quad  + 2 Q(r) \big ( \int_0^r K(r,s) Q(s) v(s) \, ds \big ),
\end{align}
where $e > 0$ and $K(r,s)$ are the same as in (\ref{eq:ee}) and (\ref{eq:Kdef}), respectively. Furthermore, the function $V(r)$ is defined by
\begin{equation}
V(r) = -\big ( |x|^{-2} \ast |Q|^2)(r) .
\end{equation}
As previously noted in the proof of Theorem \ref{th:Qunique}, we can henceforth assume that 
\[ e= 1, \]
which follows by rescaling $Q(r) \mapsto e^{-1} Q(e^{-1/2} r)$ if $e \neq 1$. (Likewise, the operator $L_+$ changes, but all its kernel elements are obtained by rescaling also.)

To proceed with the proof of Theorem \ref{th:ker}, we need the following auxiliary result.

\begin{lemma} \label{lem:grow}
Suppose $\mathcal{L}_+ v = 0$ with $v \not \equiv 0$ and $v'(0) = 0$. Then
\[
|v(r)| \gtrsim e^{+\delta r}, \quad \mbox{for $r \geq R$},
\]
where $\delta > 0$ and $R  >0$ are some suitable constants.
\end{lemma}

\begin{proof}[Proof of Lemma \ref{lem:grow}] First, we note that $v(0) \neq 0$ holds, since otherwise $v \equiv 0$, by the local uniqueness for the linear differential-integro equation $\mathcal{L_+} v = 0$, which follows from standard fixed point argument. 

Since $\mathcal{L}_+ v = 0$ is a linear equation, we can assume without loss of generality that 
\[ v(0) > Q(0) . \]
Next, we note that $v$ satisfies
\begin{equation} \label{eq:v}
-v''(r) - \frac{3}{r} v'(r) +  v(r) + V(r) v(r) + W(r) = 0,
\end{equation} 
where we set
\begin{equation} \label{eq:W}
W(r) = 2 Q(r) \int_0^r K(r,s) Q(s) v(s) \, ds.
\end{equation}
Clearly, the ground state $Q(r)$ satisfies
\begin{equation} \label{eq:Q}
-Q''(r) - \frac{3}{r} Q'(r) + Q(r) + V(r) Q(r) = 0 .
\end{equation}
Similar as in the proof of Theorem \ref{th:Qunique}, we find by using equations (\ref{eq:v}) and (\ref{eq:Q}) the ``Wronskian-type'' identity
\begin{equation}
r^3 ( Q v' - Q' v) = \int_0^r s^3 Q(s) W(s) \, ds.
\end{equation}
Hence we conclude (keeping in mind that $Q(r) > 0$) the identity
\begin{equation} \label{eq:ratio}
r^3 \Big ( \frac{v(r)}{Q(r)} \Big )' = \frac{1}{Q(r)^2} \int_0^r s^3 Q(s) W(s) \, ds.
\end{equation}
Since $Q(r) > 0$ for all $r \geq 0$ and $v(r) > 0$ at least initially, by continuity, we see from (\ref{eq:W}) that $W(r) > 0$ for $r >0$ at least initially. Therefore, by (\ref{eq:ratio}), we have that $(v/Q)' > 0$ for $r > 0$ at least initially, and thus
\begin{equation}
v(r) > Q(r), \quad \mbox{for $r \geq 0$ at least initially.}
\end{equation} 
Again, by a similar argument as in the proof of Theorem \ref{th:Qunique}, we conclude from (\ref{eq:ratio}) that $v(r)$ and $Q(r)$ never intersect. Hence,
\begin{equation}
v(r) > Q(r), \quad \mbox{for $r \geq 0$.}
\end{equation}
Furthermore, this fact combined with equation (\ref{eq:ratio}) yields the lower bound
\begin{equation} \label{eq:ratio2}
r^3 \Big ( \frac{v(r)}{Q(r)} \Big )'  \geq \frac{2}{Q(r)^2} \int_0^r s^3 Q(s)^2 \int_0^s K(s,t) Q(t)^2 \, dt \, ds , 
\end{equation}
for $r \geq 0$. 

Next, we note that $Q(r)$ is the (unique) ground state eigenfunction for the Schr\"odinger operator $H = -\DD + V$, so that
\begin{equation}
H Q = -\omega Q, \quad \mbox{where $-\omega = \inf \sigma(H) < 0$}.
\end{equation}
(Note that in principle $\omega \neq 1$, because we rescaled the ground state $Q(r)$.) By well-known results on ground states for Schr\"odinger operators (see, e.\,g., \cite{Carmona+Simon1981}), we deduce that, for any $\varepsilon > 0$, there are constants $A_\varepsilon > 0$ and $B_\varepsilon > 0$ such that
\begin{equation} \label{ineq:Carmona}
A_\varepsilon e^{-(\omega + \varepsilon)r} \leq Q(r) \leq B_\varepsilon e^{-(\omega - \varepsilon) r}, \quad \mbox{for $r \geq 0$.}
\end{equation}
Now choose $0 < \varepsilon < \omega$. By inserting the bounds (\ref{ineq:Carmona}) into (\ref{eq:ratio2}), we deduce
\begin{equation} \label{eq:ratio3}
r^3 \Big ( \frac{v(r)}{Q(r)} \Big )'  \geq C e^{(2 \omega - 2 \varepsilon) r} \int_0^r s^3 e^{-(2 \omega + 2 \varepsilon) s} \int_0^s K(s,t) e^{-(2\omega +2\varepsilon) t} \, dt \, ds , \end{equation}
for all $r \geq 0$ and some constant $C > 0$ (whose dependence on $\varepsilon$ we drop henceforth). Since the nonnegative integral on the right side converges as $r \to \infty$ to some value $2A > 0$, say, there is some $R > 0$ such that
\[
\int_0^r s^3 e^{-(2 \omega + 2 \omega) s} \int_0^s K(s,t) e^{-(2 \omega + 2 \varepsilon)  t}  \, dt \, ds \geq A , \quad \mbox{for $r \geq R$.}
\]
Inserting this bound into (\ref{eq:ratio2}) yields 
\begin{equation}
r^3 \left ( \frac{v(r)}{Q(r)} \right ) '  \geq C e^{(2 \omega  - 2 \varepsilon) r}, \quad \mbox{for $r \geq R$},
\end{equation}
with some constants $C > 0$. Integrating this lower bound and using (\ref{ineq:Carmona}) once more, we  conclude that
\begin{equation}
 v(r) \geq C \frac{e^{(\omega - 3 \varepsilon) r}}{r^3} , \quad \mbox{for $r \geq R$},
\end{equation}
for some $ C>0$ and $R>0$ sufficiently large, whence Lemma \ref{lem:grow} follows by choosing $\varepsilon > 0$ sufficienlty small. \end{proof}

We now return to the proof of Theorem \ref{th:ker}. From equation (\ref{eq:LL}) and $L_+ \xi = 0$ we deduce that $\xi$ solves the inhomogeneous problem
\begin{equation}
\mathcal{L}_+ \xi =  \sigma Q, \quad \mbox{with} \quad \sigma = 2 \int_{\RR^4} \frac{Q \xi}{|x|^2} \, dx.
\end{equation}
Hence, we have that
\begin{equation} \label{eq:xii}
\xi = v + w,
\end{equation}
where $v$ solves the homogeneous equation $\mathcal{L}_+ v = 0$, and where $w$ is a particular solution to the inhomogeneous problem $\mathcal{L}_+ w = \sigma Q$. To construct $w$, we notice that  a calculation yields
\begin{equation}
L_+ Q_1 = -2Q, \quad \mbox{where $Q_1 = 2Q + r \partial_r Q$}.
\end{equation}
Thus, by using equation (\ref{eq:LL}) with $Q_1$ in place of $\xi$, we find
\begin{equation}
\mathcal{L}_+ Q_1 = -2 (1 - \rho) Q, \quad \mbox{with} \quad \rho = \int_{\RR^4} \frac{Q Q_1}{|x|^2} \,dx .
\end{equation}
Since $Q_1 \in L^2_{\mathrm{rad}}$ and $Q_1(0) \neq 0$ with $Q_1'(0)=0$, we conclude from Lemma \ref{lem:grow} that $\mathcal{L}_+ Q_1 \neq 0$ and hence $\rho \neq 1$ must hold. Therefore, we have found a particular solution $w$ given by
\begin{equation} \label{eq:w}
w = \frac{\sigma}{1- \rho} Q_1 \in L^2_{\mathrm{rad}}(\RR^4).
\end{equation}

Next, suppose that $\sigma = 0$ so that $w =0$. Then equation (\ref{eq:xii}) implies that $\xi = v$ and, by the smoothness of $\xi \not \equiv 0$, we deduce that $\xi'(0) = v'(0) = 0$. By Lemma \ref{lem:grow}, we deduce that $\xi = v \not \in L^2_{\mathrm{rad}}$, which is a contradiction.

Hence $\sigma \neq 0$ holds in (\ref{eq:w}) and therefore $w \not \equiv 0$. This, in turn, implies that $v \not \equiv 0$ must hold in (\ref{eq:xii}); for otherwise we would have $\xi \sim Q_1$ contradicting $L_+ \xi = 0$ and $L_+ Q_1 =-2Q$. Thus we can invoke Lemma \ref{lem:grow} again to find that $v \not \in L^2_{\mathrm{rad}}$. But this contradicts (\ref{eq:xii}) and that $\xi$ and $w$ are both in $L^2_{\mathrm{rad}}$. This completes the proof of Theorem \ref{th:ker}. \end{proof}

\subsection{Spectral Properties of $\Hinf$} \label{subsec:Hinf}

In this brief subsection, we prove Lemma \ref{lem:H_inf}, by using the nondegeneracy result about $L_+$ (i.\,e.~Theorem \ref{th:ker}) combined with standard arguments for NLS with local nonlinearities.

\begin{proof}[Proof of Lemma \ref{lem:H_inf}]
To see that $\sigma_{\rm ess}(\Hinf) = (-\infty, -1] \cup [1, \infty)$ holds, one argues (using Weyl's lemma and the fact that the local terms vanish at infinity) in the same way as for linearized operator for ground states of NLS with local nonlinearities; see, e.\,g., \cite{Erdogan+Schlag2006,Hundertmark+Lee2007}. Note that the nonlocal term
\begin{equation}
2 \Qinf \big ( |x|^{-2} \ast (\Qinf \cdot ) \big )
\end{equation}
is easily seen to be a Hilbert-Schmidt operator and hence compact. In particular, this operator does not affect the essential spectrum. This shows Part (i).

Thanks to the fact $\mathrm{ker} \, L_+ = \{ 0 \}$ in the radial sector by Theorem \ref{th:ker}, we conclude that the generalized null-space of $\Hinf$ (in the radial sector) is given by
\begin{equation}
\mathcal{N} = \mathrm{span} \, \{ \phi_1, \phi_2, \phi_3, \phi_4 \},
\end{equation}
with $\{ \phi_i \}_{i=1}^4$ as in Lemma \ref{lem:H_inf}, by an immediate adaptation of \cite{Weinstein1985} which proves this fact for NLS with $L^2$-critical nonlinearities. Next, let $\rho \in L^2_{\rm rad}(\RR^4)$ satisfy
\begin{equation}
L_+ \rho = - |x|^2 \Qinf .
\end{equation}
Note that $\rho$ exists and is unique, since $L_+$ is invertible in the radial sector. Furthermore, we see that $e^{+\delta |\cdot|} \rho \in L^\infty$ for some $\delta > 0$, by adapting, e.\,g., the proof of \cite[Lemma 4.9]{FJL2007}. (Note that $L_+$ is a nonlocal operator, so we cannot directly use standard arguments to deduce exponential decay). Finally, we mention that Part (iii) is a well-known fact for linearized operators in the context of local NLS, and the proof carries over without modification. This completes the proof of Lemma \ref{lem:H_inf}. \end{proof}

\section{Construction and Properties of $\Qt$} \label{sec:Qt}

In the section, we construct radial and real-valued solutions $\Qt \in H^1(\RR^4)$ of
\begin{equation} \label{eq:pertQ}
-\DD \Qt + \Qt - \Big ( \frac{ \phi(t^{-k} | \cdot |^k)}{| \cdot |^2} \ast |\Qt|^2 \Big ) \Qt = 0,
\end{equation}
for $t \geq T_0$, where $T_0 > 0$ is sufficiently large and $k >0$ denotes some fixed number. Here $\phi(r)$ is supposed to satisfy the assumptions stated in Theorem \ref{th:main}. Note that, if we formally set $t=\infty$ in \eqref{eq:pertQ}, we obtain that $\Qinf$ should satisfy
\begin{equation} \label{eq:Qclear}
-\DD \Qinf + \Qinf - \big ( \frac{1}{| x |^2} \ast |\Qinf|^2 \big ) \Qinf = 0 ,
\end{equation}
where $\Qinf = Q \in H^1(\RR^4)$ is the unique ground state given by Theorem \ref{th:Qunique}.

\subsection{Construction of $\Qt$}

We now construct solutions $\Qt$ of (\ref{eq:pertQ}) by an implicit-function-type argument such that $\lim_{t \rightarrow \infty} \Qt = \Qinf$ in $H^1$. This construction is essentially based upon the nondegeneracy result stated in Theorem \ref{th:ker} above. Since $\Qt$ are radial and real-valued functions in $H^1(\RR^4)
$, it is convenient to introduce the following subspace
\begin{equation}
\Honer =\big  \{ f \in H^1(\RR^4) : \mbox{$f$ is radial and real-valued} \big \} . 
\end{equation}
We have the following result.

\begin{theorem} \label{th:Qt}
Let $k >0$ be fixed and suppose that $\phi(r)$ satisfies the assumptions stated in Theorem \ref{th:main}. Furthermore, let $\Qinf \in \Honer$ denote the unique ground state solving (\ref{eq:Qclear}) given by Theorem \ref{th:Qunique}. Then there exists $T_0 = T_0(k) > 0$ sufficiently large such that the following properties hold.
\begin{enumerate}
\item[(i)] There exists a unique $C^0$-map 
\[
t \mapsto \Qt, \quad [T_0, \infty) \to \Honer ,
\]
such that $\Qt \in \Honer$ solves (\ref{eq:pertQ}) and $\lim_{t \rightarrow \infty} \| \Qt - \Qinf \|_{H^1_x} = 0$.
\item[(ii)] The map $t\mapsto \Qt$ is $C^1$ and satisfies
\[
\big \| \partial_t \Qt \big \|_{H^1_x} \lesssim t^{-1}  , \quad \mbox{for $t \geq T_0$}.
\]
\end{enumerate} 
\end{theorem}

\begin{proof}
For $\tau \geq 0$ and $u \in \Honer$, we define the following map
\begin{equation}
G(u,\tau) = u + (-\DD + 1)^{-1} g(u,\tau),
\end{equation}
where we set
\begin{equation}
g(u,\tau) = - \big ( \frac{\phi(\tau^k |\cdot|^k)}{|\cdot|^2} \ast |u|^2 \big ) u. 
\end{equation}
Clearly, any $(Q_\tau,\tau)$ such that $G(Q_\tau,\tau) = 0$ is a solution of equation (\ref{eq:pertQ}) with $t = 1/\tau$, provided that $\tau > 0$. Also, the limit behavior of $Q_\tau$ as $\tau \to 0$ obviously corresponds to the limit $t \to \infty$.

Using an implicit function theorem and Theorem \ref{th:ker}, we shall first construct, for $\tau_0 > 0$ sufficiently small, a unique $C^0$-map 
\begin{equation} \label{eq:phimap}
\tau \mapsto \varphi(\tau), \quad [0, \tau_0] \to \Honer,
\end{equation}
such that $G(\varphi(\tau), \tau) = 0$ for all $\tau \in [0,\tau_0]$ and $\varphi(0) = \Qinf$. To this end, let $\tau_0 > 0$ be a constant chosen later. By Hardy's inequality, we find that $g : H^1_{\mathrm{rad}} \times [0,\tau_0] \to H^1_{\mathrm{rad}}$ holds. Also, it is easy to see that $\partial_u g(u,\tau)$ exists and is given by
\begin{equation} \label{eq:gu}
\partial_u g (u,\tau) = - \big ( \frac{\phi(\tau^k |\cdot |^k)}{| \cdot |^2} \ast |u|^2 \big ) - 2 u \big (  \frac{\phi(\tau^k |\cdot |^k)}{| \cdot |^2} \ast (u \cdot ) \big  ) .
\end{equation}
Next, we claim that 
\begin{equation}
\partial_u G(u,\tau) = 1 + (-\DD + 1)^{-1} \partial_u g(u,\tau)
\end{equation}
depends continuously on $(u,\tau)$ in the $H^1$-topology, i.\,e.,
\begin{equation}
\big \| ( \partial_u G(u_1,\tau_1)  - \partial_u G(u_2, \tau_2) )  f \big \|_{H^1} \to 0 \quad \mbox{as} \quad (u_1, \tau_1) \to (u_2, \tau_2),
\end{equation}
for all $f \in \Honer$. To show this, we first note
\begin{align}
\big \| \{ \partial_u G(u_1, \tau_1) - \partial_u G(u_2, \tau_2) \} f \big \|_{H^1}& \lesssim \big \| \frac{\sqrt{-\Delta + 1}}{-\Delta + 1} \{ \partial g(u_1, \tau_1 ) - \partial_u g(u_2, \tau_2) \} f \big \|_{L^2} \nonumber \\
& \lesssim \big \| \{ \partial_u g(u_1,\tau_1) - \partial_u g(u_2, \tau_2) \} f \big \|_{L^2}  \nonumber \\
& \lesssim \big \| \{ \partial_u g(u_1, \tau_1) - \partial_u g(u_2, \tau_1) \} f \big \|_{L^2} \nonumber \\
& \quad + \big \| \{ \partial_u g(u_2,\tau_1) - \partial_u g(u_2, \tau_2) \} f \big \|_{L^2}  \label{eq:Gpart}
\end{align}
Here the first term on the right side is easily estimated as follows (using H\"older's and Hardy's inequality) together with the fact that $|\phi| \lesssim 1$ holds. For example, the contribution from the first (local) term in (\ref{eq:gu}) is bounded by
\begin{align*}
\| \big ( \frac{\phi(\tau^k |\cdot|^k)}{|\cdot|^2} \ast (|u_1|^2 - |u_2|^2) \big ) f \big \|_{L^2} & \lesssim \| \big ( \frac{\phi(\tau^k |\cdot|^k)}{|\cdot|^2} \ast (|u_1|^2 - |u_2|^2) \big ) \|_{L^\infty} \|f \|_{L^2} \\
& \lesssim \sup_{y \in \RR^d} \Big | \int_{\RR^d} \frac{|u_1(x) + u_2(x)|^{2}}{|x-y|^2} \Big |^{1/2}  \\
& \quad \cdot \sup_{y \in \RR^d} \Big | \int_{\RR^d} \frac{|u_1(x) - u_2(x)|^{2}}{|x-y|^2} \Big |^{1/2} \cdot \| f \|_{L^2} \\
& \lesssim  ( \| u_1 \|_{H^1} + \| u_2 \|_{H^1} ) \| u_1 - u_2 \|_{H^1} \| f \|_{L^2} .
\end{align*}
Likewise, we can estimate the part arising for the second (nonlocal) term in $\partial_u g$. In summary, we have 
\begin{equation}
\big \| \{ \partial_u g(u_1, \tau_1) - \partial_u g(u_2, \tau_1) \} f \big \|_{L^2} \lesssim ( \| u_1 \|_{H^1} + \| u_2 \|_{H^1} ) \| u_1 - u_2 \|_{H^1} \| f \|_{L^2} .
\end{equation}

To deal with the second term on the right side of (\ref{eq:Gpart}), we just apply the dominated convergence to conclude that
\begin{equation}
\big \| \{ \partial_u g(u_2, \tau_1) - \partial_u g(u_2, \tau_2) \} f \big \|_{L^2} \to 0 \quad \mbox{as} \quad \tau_1 \to \tau_2 .
\end{equation}
This proves the claimed continuity of $\partial_u G (u,\tau)$. Next, we note that $(\Qinf, 0)$ satisfies
\begin{equation}
G(\Qinf, 0) = 0 .
\end{equation}
Furthermore, by Theorem \ref{th:ker}, we have that $L_+ = \{ 0 \}$ when acting on radial functions. This that the compact operator $(-\Delta + 1)^{-1} g_u(\Qinf, 0)$ does not have $-1$ in its spectrum. Hence $\partial_u G(\Qinf,0)$ is invertible on $\Honer$. Therefore, by an implicit function argument, we deduce the existence of the unique $C^0$-map (\ref{eq:phimap}), provided that $\tau_0 > 0$ is chosen sufficiently small. By defining $\Qt = \varphi(1/t)$ for $t \in [T_0, \infty)$ with $T_0 = 1/\tau_0$, we complete the proof of Part (i).

To show Part (ii), we observe that the derivative $\partial_\tau G(u, \tau)$ exists and is continuous, provided that $0 < \tau \leq \tau_0$. Thus, 
\begin{equation}
\partial_\tau \varphi(\tau) = - (\partial_u G(\varphi(\tau), \tau)^{-1} G_\tau(\varphi(\tau), \tau)
\end{equation}
exists and is continuous for $0 < \tau \leq \tau_0$. By calculation, we find that $ \| \partial_\tau \varphi \|_{H^1_x} \lesssim O(\tau^{-1})$ for $\tau \in (0, \tau_0]$. Since $\Qt = \varphi(1/t)$, by definition, the chain rule yields
\begin{equation}
\big \| \partial_t \Qt \big \|_{H^1_x} \lesssim O(t^{-1}), \quad \mbox{for $t \in [T_0, \infty)$.}
\end{equation}
which completes the proof of Theorem \ref{th:Qt}. \end{proof}

\subsection{Spectral Properties of $\Ht$} \label{subsec:Ht}

The proof of Lemma \ref{lem:H_t} is based on Lemma \ref{lem:H_inf} and the convergence $\Qt \to \Qinf$ in $H^1$ as $t \to \infty$.

\begin{proof}[Proof of Lemma \ref{lem:H_t}] 
The claim that $\sigma_{\mathrm{ess}}(\Ht) = (-\infty, -1] \cup [1, \infty)$ follows from the same argument as for Lemma \ref{lem:H_inf}. 

It remains to prove Part (ii). To this end, we define the Riesz projection
\begin{equation}
\Prinf = \frac{1}{2 \pi i} \oint_{|z| = c} (z-\Hinf)^{-1} \, dz,
\end{equation}
where $c > 0$ is chosen sufficiently small such that $0$ is the only eigenvalue of $\Hinf$ inside the disc $|z| \leq c$. Then $\Prinf$ projects onto the generalized null-space $\mathcal{N}$ of $\Hinf$. Next, we note that it is easy to see that the difference
\begin{equation}
\delta \Ht := \Hinf - \Ht
\end{equation}
is a bounded operator on $L^2_{\mathrm{rad}}(\RR^4; \mathbb{C}^2)$ with operator norm $\| \delta \Ht \| \to 0$ as $t \to \infty$, by using that $\Qt \to \Qinf$ in $H^1$ as $t \to \infty$. Therefore, by choosing $c > 0$ possibly smaller and $T_0 > 0$ sufficiently large, we readily verify that the Riesz projections
\begin{equation}
\Prt = \frac{1}{2 \pi i} \oint_{|z| = c} (z-\Ht)^{-1} \, dz 
\end{equation}
exists for all $t \in [T_0, \infty)$. Furthermore, we have $\| \Prt -  \Prinf \|_{L^2_x \to L^2_x} \to 0$ as $t \to \infty$. Furthermore, it is straightforward to verify that indeed we have
\begin{equation}
\| \Prt - \Prinf \|_{H^1_x \to H^1_x} \quad \mbox{as} \quad t \to \infty,
\end{equation}
using that $\| \delta \Ht \|_{H^1_x \to H^1_x} \to 0$ as $t \to \infty$.  The proof of Lemma \ref{lem:H_t} is complete. \end{proof}

\subsection{Improved Estimate for $\pr_t \Qt$}

We now show that needed stronger decay estimate
\begin{equation}
\pr_t \Qt = \mathcal{O}(t^{-k-1}), \quad \mbox{for $t \gg 1$}.
\end{equation} 
To this end, we first prove the following lemma about uniform exponential decay.
\begin{lemma} \label{lem:expdecay}
Let $\{ \Qt \}_{t \in [T_0, \infty)}$ be as in Theorem \ref{th:Qt}. Then there exist constants $\delta > 0$ and $C > 0$ such that
\[
| \Qt(x) | \leq C e^{- \delta |x|}, \quad \mbox{for all $x \in \RR^4$ and all $t \in [T_0, \infty)$}.
\]
\end{lemma}

\begin{proof} 
For each $t \in [T_0, \infty)$, we have that $\Qt$ is an eigenfunction with eigenvalue $-1$ for the Schr\"odinger operator $H_t$, i.\,e.,
\begin{equation} \label{eq:Slaggie1}
H_t Q_t = -Q_t,
\end{equation}
where
\begin{equation}
H_t = -\DD - \Vt, \quad \mbox{with $\Vt =  \big ( \frac{\phi(t^{-k} | \cdot |^k)}{| \cdot |^2} \ast |\Qt|^2 \big )$. }
\end{equation}
In what follows, we adapt the Slaggie-Wichmann method (see, e.\,g., \cite{Hislop2000})  to prove pointwise exponential decay for $\Qt$,  where we have to make sure that all the constants can be chosen uniform in $t$. Therefore, we need to detail the proof as follows.

First, we rewrite (\ref{eq:Slaggie1}) as
\begin{equation} \label{eq:Slaggie2}
\Qt(x) = \int_{\RR^4} G(x-y) \Vt(y) \Qt(y) \, dy,
\end{equation} 
where $G(x-y)$ is the kernel of the resolvent $(-\DD + 1)^{-1}$. Next, let $0 < \delta < 1$ be a fixed number and define the functions
\begin{equation} \label{def:ht}
h^{(t)}(x) = \int_{\RR^4} e^{\delta |x-y|} |G(x-y)| |\Vt(y)| \,dy .
\end{equation}
By well-known estimates for the resolvent kernel and our choice that $\delta < 1$, we see that $e^{+\delta |z|} G(z) \in L^1(\RR^4) + L^p(\RR^4)$ for any $p \geq 1$. Since $\Vt \in L^{2+}(\RR^4) \cap L^\infty(\RR^4)$, one easily checks that $h^{(t)} \in C^0(\RR^4)$ and $\lim_{|x| \to \infty} h^{(t)}(x) = 0$.

Next, we define the functions
\begin{equation}
m^{(t)}(x) = \sup_{y \in \RR^4} |\Qt(y)| e^{-\delta |x-y|},
\end{equation}
which, in view of (\ref{eq:Slaggie2}), leads to the inequality
\begin{equation} \label{ineq:Slaggie3}
|\Qt(x)| \leq h^{(t)} (x) m^{(t)} (x) .
\end{equation}
Now let $R > 0$ be a constant specified below. Clearly, we have that
\begin{equation} \label{eq:mmax}
m^{(t)} (x) = \max \Big \{ \sup_{|y| \leq R} |\Qt(y)| e^{-\delta |x-y|}, \sup_{|y| > R} |\Qt(y)| e^{-\delta |x-y|} \Big \} .
\end{equation}
Our goal is now to show that if $R > 0$ is sufficiently large, then the maximum in (\ref{eq:mmax}) is always given by the $\sup_{|y| \leq R}$-term for all $t \geq T_0$. The claimed uniform exponential decay for $\Qt$ then follows easily.

We now show that there is indeed such an $R >0$. To this end, we first claim that we can take $R > 0$ sufficiently large such that
\begin{equation} \label{ineq:ht}
h^{(t)}(x) < \frac{9}{10} , \quad \mbox{for $|x| > R$ and $t \geq T_0$} .
\end{equation}
Indeed, since $|\phi(x)| \lesssim 1$, we deduce $|\Vt(x)| \lesssim |x|^{-2}$, by using Newton's theorem and the fact that $\Qt$ are radial functions with $\| \Qt \|_{L^2_x} \lesssim1$. Furthermore, we have $|\Vt(x)|\lesssim 1$ by Hardy's inequality and $\| \Qt \|_{\dot{H}^1_x} \lesssim 1$. Therefore, we obtain the uniform pointwise bound
\begin{equation} \label{ineq:Vupper}
|\Vt(x)| \lesssim \frac{1}{1 + |x|^2} , \quad \mbox{for $t \geq T_0$} .
\end{equation}
Plugging this estimate into (\ref{def:ht}), we deduce the uniform bound
\begin{equation} \label{ineq:ht2}
h^{(t)}(x) \lesssim \int_{\RR^4} e^{\delta|x-y|} |G(x-y)| \frac{1}{1+|y|^2} \, dy, \quad \mbox{for $t \geq T_0$},
\end{equation}
whence (\ref{ineq:ht}) follows by taking $R >0$ sufficiently large. Furthermore, we note that equation (\ref{eq:pertQ}) combined with a bootstrap arguments shows that $\| \Qt \|_{H^1_x} \lesssim 1$ yields $\| \Qt \|_{H^s_x}\lesssim C(s) $ for any $s \geq 1$ and some constants $C(s)$. In particular, by Sobolev's embedding, we conclude the uniform bound
\begin{equation} \label{ineq:Qtupper}
 \| \Qt  \|_{L^\infty_x} \lesssim 1 , \quad \mbox{for $t \geq T_0$}.
\end{equation}

With the help of the uniform bound (\ref{ineq:ht}) and \eqref{ineq:Qtupper}, we now deduce the claimed uniform exponential along the lines of the Slaggie-Wichmann argument. First, we observe that
\begin{equation} \label{eq:Slaggie4}
e^{-\delta |x-y|} = \sup_{z \in \RR^4} e^{-\delta |x-y|} e^{-\delta |y-z|} ,
\end{equation}
which directly follows from the triangle inequality and the definition of the supremum.  Thus, we have
\begin{equation}
m^{(t)}(x) = \sup_{z \in \RR^4} m^{(t)}(z) e^{-\delta |x-z|} .
\end{equation}
Now assume that $R > 0$ such that (\ref{ineq:ht}) holds. Then, by (\ref{ineq:Slaggie3}), we have $|\Qt(x)| < m^{(t)}(x)$ whenever $|x| > R$. This fact and (\ref{eq:Slaggie4}) imply
\begin{align}
\sup_{|y| > R} |\Qt(y)| e^{-\delta |x-y|} & < \sup_{|y| > R} m^{(t)}(y) e^{-\delta |x-y|}  \nonumber \\
& \leq \sup_{y \in \RR^4} m^{(t)}(y) e^{-\delta |x-y|} = m^{(t)}(x) . \nonumber  
\end{align}
Hence the $\sup_{|y| > R}$-term in (\ref{eq:mmax}) is strictly less than $m^{(t)}(x)$; and therefore
\begin{equation}
m^{(t)}(x) = \sup_{|y| \leq R} |\Qt(y)| e^{-\delta |x-y|} .
\end{equation}
Using the uniform bound (\ref{ineq:Qtupper}) and $\sup_{|y| \leq R} e^{\delta |y|} \lesssim 1$, we deduce that
\begin{equation}
m^{(t)}(x) \lesssim e^{-\delta |x|} , \quad \mbox{for $t \geq T_0$} .
\end{equation}
Going back to (\ref{ineq:Slaggie3}) and noting that (\ref{ineq:ht2}) also shows that $\| h^{(t)} \|_{L^\infty_x} \lesssim 1$ for $t \geq T_0$, we complete the proof of Lemma \ref{lem:expdecay}. \end{proof}

Finally, we are in the position to derive the following improved decay estimate for $\partial_t \Qt$ for $t \gg 1$.

\begin{lemma}\label{lem:tdecay}
For $\{ \Qt \}_{t \in [T_0,\infty)}$ as in Theorem \ref{th:Qt} and $T_0 > 0$ sufficiently large, we have
\[
\| \partial_t \Qt \|_{H^1_x} \lesssim t^{-k-1}, \quad \mbox{for $t \geq T_0$}.
\]
\end{lemma}

\begin{proof}
Note that $\Qt$ satisfies
\begin{equation} \label{eq:GQ}
G(\Qt, t) = 0, \quad \mbox{for all $t \in [T_0, \infty)$},
\end{equation}
where $G : \Honer \times [T_0, \infty) \to \Honer$ is given by
\begin{equation}
G(u,t) = u + (-\DD + 1)^{-1} g(u,t), \quad \mbox{with} \quad g(u,t) = - \big ( \frac{\phi(t^{-k} |\cdot|^k)}{ | \cdot |^2} \ast |u|^2 ) u .
\end{equation}
Differentiating equation (\ref{eq:GQ}) with respect to $t$ yields
\begin{equation}
\partial_t \Qt = - \big ( \partial_u G(\Qt, t) \big )^{-1} \partial_t G(\Qt,t)
\end{equation}
for all $t$ sufficiently large, while using the fact that $\partial_u G(\Qt,t)$ is invertible for $(\Qt,t)$ close $(\Qinf, \infty)$. Furthermore, by continuity, we have that 
\[ \| \partial_u G(\Qt,t)^{-1} \|_{H^1 \to H^1} \lesssim 1, \]
for $t$ sufficiently large. Using that $\partial_t G(\Qt,t) = (-\DD + 1)^{-1} \partial_t g(\Qt,t)$, we therefore get
\begin{align*}
\| \partial_t \Qt \|_{H^1} & \lesssim  \| (-\DD + 1)^{-1} \partial_t g(\Qt,t) \|_{H^1} \lesssim \| \partial_t g (\Qt,t) \|_{L^2}
\end{align*}
Next, we note
\begin{equation}
(\partial_t g(\Qt,t))(x) = -k t^{-k-1} \Big ( \int_{\RR^d} \frac{\phi'(t^{-k} |x-y|^k)}{|x-y|^2} |x-y|^{k} |\Qt(y)|^2 \, dy \big ) \Qt(x) .
\end{equation}
Let now $m=k-2$. If $-2 \leq m \leq 0$ then, by Young's and Hardy's inequality and using that $|\phi'(r)| \lesssim 1$, we obtain
\begin{equation}
\big \| \frac{\phi'(t^{-k} |\cdot|^k)}{ | \cdot |^{-m}} \ast |\Qt|^2 \big \|_{L^\infty} \lesssim 1 .
\end{equation}
Hence we conclude that
\begin{equation}
\label{estg}
\| \partial_t g_t(\Qt,t) \|_{L^2} \lesssim t^{-k-1} \| \Qt \|_{L^2} \lesssim t^{-k-1} .
\end{equation}
whenever $m=k-2 \in (-2,0]$, i.\,e., for $k \in (0,2]$. It remains to prove such a bound when $k > 2$. To this end, we use the uniform exponential decay stated in Lemma 1. For $m = k-2 > 0$, we recall the elementary inequality 
\[ |x-y|^m \leq \max \{ 1, 2^{m} \} ( |x|^m + |y|^m ) .\]
Next, by using Lemma 1, we deduce the pointwise bound
\begin{align*}
| \pr_t g(\Qt,t) |(x) & \lesssim t^{-k-1} \Big ( \int_{\RR^d} (|x|^m + |y|^m) e^{- 2 \delta |y|} \, dy \Big ) e^{- \delta |x|} \\
& \lesssim t^{-k-1} ( |x|^m + 1) e^{-\delta |x|} ,
\end{align*}
for some $\delta > 0$. This shows that $\| \pr_t  g (\Qt,t) \|_{L^2} \lesssim t^{-k-1}$ for $t$ large and if $k > 2$. This completes the proof of Lemma \ref{lem:tdecay}. \end{proof}

\bibliographystyle{amsalpha}

\bibliography{HartreeBib.bib}

\end{document}